\documentclass[12pt]{article}
\usepackage{mathrsfs}
\usepackage{amsmath}
\usepackage{float}
\usepackage{cite}
\usepackage{esint}
\usepackage[all]{xy}
\usepackage{amssymb}
\usepackage{amsfonts}
\usepackage{amsthm}
\usepackage{color}
\usepackage{graphicx}
\usepackage{hyperref}
\usepackage{bm}
\usepackage{indentfirst}
\usepackage{geometry}
\usepackage{authblk}
\theoremstyle{plain}
\newtheorem{thm}{Theorem}[section]

\newtheorem{defn}[thm]{Definition}
\newtheorem{prop}[thm]{Proposition}
\newtheorem{conj}[thm]{Conjecture}
\newtheorem{cor}[thm]{Corollary}
\newtheorem{lem}[thm]{Lemma}

\newtheorem{them}[thm]{Theorem}
\theoremstyle{remark}
\newtheorem{ex}[thm]{Example}
\newtheorem{rmk}[thm]{Remark}

\newcommand{\R}{\mathbb{R}}

\newcommand{\Hom}{\mathrm{Hom}}
\newcommand{\CF}{\mathrm{CF}}
\newcommand{\im}{\mathrm{Im}}

\newcommand{\Sym}{\mathrm{Sym}}
\newcommand{\coker}{\mathrm{coker\,}}
\newcommand{\Span}{\mathrm{Span}}

\DeclareMathOperator{\rank}{rank}

\numberwithin{equation}{section}

\title{Singularities of Lagrangian Immersions and its Applications in Lagrangian Floer Theory}
\author{Zuyi Zhang \\ \href{mailto:zuyzhang@iu.edu}{zuyzhang@iu.edu}}
\affil{Department of Mathematics\\ Indiana University Bloomington\\ Rawles Hall 831 East 3rd St. Bloomington, IN 47405-7106}
\date{ }

\begin{document}

\maketitle

\begin{abstract}
    In this article, we study singularities of Lagrangian immersions into Cartesian product of surfaces. After applying a {regular homotopy} in the Weinstein tubular neighbourhood of the Lagrangian immersion, the singular points of the Lagrangian immersion can be expressed locally as fold points with finitely many cusp points. This result has {potential} applications to comparing the Lagrangian Floer complexes associated to curves on surfaces related by a Lagrangian correspondence.
\end{abstract}

\section*{Acknowledgement}
The author offers his sincerest gratitude to his advisor Professor Paul Kirk, for his numerous support in this project. The author also would like to thank Christopher Herald, Zhangkai Huang, and the anonymous referee for the helpful suggestions after reading the paper.

\section{Introduction}
In this article, the behavior of Lagrangian immersions from surfaces to closed symplectic 4-manifolds and its application to Lagrangian Floer theory are studied. The article consists of two parts. 
\begin{itemize}
    \item The first part consists of the proof of a theorem (Theorem \ref{thm:main}) describing the generic local structure of a Lagrangian immersion from a surface into a Cartesian product of two surfaces (Section \ref{sec:p} -- Section \ref{sec:immerofbisin}) and the behavior of the bisingular set (Definition \ref{def:bisin}) of the Lagrangian immersion (Section \ref{sec:trabisin}).
    \item The second part {discusses possible applications} of this theorem {(Section \ref{sec:Lagcomp}-Section \ref{sec:Lagcov})}. The motivation of this article is to establish a relation of two Lagrangian Floer complexes related by a Lagrangian correspondence with embedded bisingularities (Definition \ref{def:bisin}). The project has not yet been completed. {Suppose the Lagrangian immersions on two surfaces are related by a Lagrangian correspondence. The main theorem in part two is to compare the bigons connecting the intersections of the Lagrangian immersions on these surfaces.} We list some examples in {the last} section. These examples indicate a possible way to determine the quilted Floer complex associated to these two Lagrangian Floer complexes.
\end{itemize}

In the first part, the symplectic 4-manifold is fixed as the Cartesian product of two closed surfaces. Let $(F_1,\omega_1)$ and $(F_2,\omega_2)$ be two closed symplectic surfaces. Then $F_1\times F_2$ is a symplectic manifold with the symplectic form $\omega_1\times (-\omega_2)$.

\begin{defn}
Let $f:F\rightarrow\Tilde{F}$ be a smooth map between surfaces. 
\begin{itemize}
    \item A point $x\in F$ is called a {\bf singular point} of $f$ if $df_x$ has rank strictly smaller than two.
    \item A critical point $x\in F$ is called a {\bf fold point} if there are local coordinates $x_1,x_2$ centered at $x$ and $y_1,y_2$ centered at $f(x)$ such that $f$ is given by
    \[
    (x_1,x_2)\mapsto(x_1, x_2^2).
    \]
    \item A critical point $x\in F$ is called a {\bf cusp point} if there are local coordinates $x_1,x_2$ centered at $x$ and $y_1,y_2$ centered at $f(x)$ such that $f$ is given by
    \[
    (x_1,x_2)\mapsto(x_1,x_1x_2+x_2^3).
    \]
\end{itemize}
\end{defn}

\begin{defn}\label{def:bisin}
Let $(F_1,\omega_1)$ and $(F_2,\omega_2)$ be two closed symplectic surfaces. Let $g=(g_1,g_2):F\rightarrow F_1\times F_2$ be a smooth Lagrangian immersion. 
\begin{itemize}
    \item A point $x\in F$ is called {\bf bisingular} if $x$ is a critical point for both $g_1$ and $g_2$.
    \item A bisingular point $x\in F$ is called a {\bf bifold point} if $x$ is a fold point for both $g_1$ and $g_2$.
    \item A bisingular point $x\in F$ is called a {\bf cusp point} if $x$ is a cusp point for either $g_1$ or $g_2$.
\end{itemize}
\end{defn}

The following is the main theorem in the first part:
\begin{them}\label{thm:main}
Let $(F_1,\omega_1)$ and $(F_2,\omega_2)$ be two closed symplectic surfaces. Suppose $g=(g_1,g_2):F\rightarrow F_1\times F_2$ is a smooth Lagrangian immersion. Then there is a positive number $\delta>0$ and a regular homotopy of Lagrangian immersions $g^t:F\rightarrow F_1\times F_2$ for $t\in[0,\delta]$ such that the followings hold.
\begin{itemize}
    \item The homotopy starts at $g$, i.e. $g^t|_{t=0}=g$.
    \item The homotopy $g^t$ is a smooth Lagrangian immersion for all $t\in(0,\delta]$ and the bisingular set of $g^\delta$ consists of bifold points and finitely many cusp points.
\end{itemize}
\end{them}
{The regular homotopy considered here is only locally Hamiltonian. More precisely, a regular homotopy becomes a Hamiltonian isotopy when restricted to sufficiently small neighborhoods. However, in cases where a Lagrangian Floer chain group forms a chain complex, regular homotopies, unlike Hamiltonian isotopies, do not necessarily preserve the Lagrangian Floer homology.}

The proof of Theorem \ref{thm:main} uses the theory of jet bundles. In section \ref{sec:trabisin}, we show that a further perturbation can be performed such that the restriction of $g_1$ and $g_2$ on the bisingular set are immersions {except at the finitely many cusp} with transversal self-intersections. 

Entov proves a similar theorem as Theorem \ref{thm:main} in \cite{entov1999surgery}, using a surgery method developed by Elia{\v{s}}berg in \cite{eliavsberg1972surgery}. Let $\pi:E\rightarrow M$ be a fiber bundle such that $E$ is a symplectic manifold and the fibers of $\pi$ are Lagrangian submanifolds. Suppose $L\subset E$ is a Lagrangian embedding. Entov projects the Lagrangian submanifold $L$ into the base space $M$ and shows that the singular points of $\pi|_L:L\rightarrow M$ are fold points after applying Hamiltonian perturbations. This is different from Theorem \ref{thm:main} in the sense that 
\begin{itemize}
    \item[$1)$] we consider the projection of Lagrangian immersions into symplectic submanifolds ($F_1$ and $F_2$ in Theorem \ref{thm:main} are symplectic submanifolds),
    \item[$2)$] we allow cusp points in the bisingular set.
\end{itemize}

Theorem \ref{thm:main} concerns Lagrangian Floer theory, which is also the motivation of the article. Let $F_1$ and $F_2$ be closed surfaces. Consider three Lagrangian immersions $L_1,{F}$, $L_2$ into $F_1,F_1\times F_2$, $F_2$, respectively. The Lagrangian immersion ${F}\looparrowright F_1\times F_2$ is called {\bf a Lagrangian correspondence}. This means that under the assumption of Proposition \ref{prop:lagco}, new Lagrangian immersions can be produced by composing $L_1$ and $L_2$ with ${F}$. More precisely, the composition of $L_1$ and ${F}$, defined as $L_1\circ {F}$, is a Lagrangian immersion in $F_2$ and the composition of ${F}$ with $L_2$, defined as ${F}\circ L_2$, is a Lagrangian immersion in $F_1$. It is natural to ask if there is any relation between the Lagrangian Floer complexes $\CF(L_1,{F}\circ L_2)$ and $\CF(L_1\circ {F},L_2)$. In Section \ref{sec:Lagcomp}, it is proved that there is a canonical identification of {the intersections of $L_1$ and $F\circ L_2$ and the intersections of $L_1\circ {F}$ and $L_2$. Under the assumptions of Abouzaid \cite{abouzaid2008fukaya}, these intersection points are generators of $\CF(L_1,{F}\circ L_2)$ and $\CF(L_1\circ {F},L_2)$, respectively.} We then compare the {bigons} connecting the intersections of $L_1$ and ${F}\circ L_2$ with the {bigon}s connecting the intersections of $L_1\circ {F}$ and $L_2$. {In the case where $F_i$ has genus greater than 1 and $L_i,F\circ L_2,L_1\circ F$ are unobstructed \cite{abouzaid2008fukaya}, $i=1,2,$} these {bigon}s correspond to the boundary maps of the Lagrangian Floer complexes $\CF(L_1,{F}\circ L_2)$ and $\CF(L_1\circ {F},L_2)$, respectively. Theorem \ref{thm:main} implies that after removing the bisingular points from $F$, the induced maps from $F$ to $F_i$, $i=1,2$, are local diffeomorphisms. In particular, if the Lagrangian immersion has no bisingular points, then the induced maps from $F$ to $F_i$, $i=1,2$, are covering maps. Theorem \ref{1} {indicates that by assuming the conditions of Abouzaid \cite{abouzaid2008fukaya}} the boundary map of $\CF(L_1,{F}\circ L_2)$ and $\CF(L_1\circ {F},L_2)$ can be identified canonically when $F\rightarrow F_i$ is a covering map for $i=1,2$. To generalize this theorem to the case where the map $F\rightarrow F_1\times F_2$ contains bisingular points, it is necessary to compare {the {bigons} connecting the intersections of $L_1$ and ${F}\circ L_2$ with the {bigon}s connecting the intersections of $L_1\circ {F}$ and $L_2$} near the bisingular points. A nontrivial example in Cazassus, Herald, Kirk, Kotelskiy \cite{cazassus2020correspondence}, which indicates the importance of studying maps with bisingular points.\\

{One of the goals of {our} project is to find examples corresponding to the Figure-eight bubblings in Bottman and Wehrheim \cite{bottman2018gromov}. In \cite{bottman2018gromov}, they proved that there is a new bubbling phenomenon called the Figure-eight bubbling when applying the strip shrinking. They conjectured that by adding the bounding cochains indicated by the Figure-eight bubbling, the Lagrangian-Floer chain groups are chain homotopic equivalent before and after the strip shrinking. Fukaya claimed that the bounding cochain constructed in \cite{fukaya2017unobstructed} is the same as the bounding cochain in \cite{bottman2018gromov}. In \cite{fukaya2017unobstructed}, he utilizes a homological algebraic argument to show the existence of the bounding cochain. More precisely, Fukaya created an equation and solved it using an inductive way. Then he proved that the solution to this equation satisfies the Maurer-Cartan equation. These are purely algebraic and we are trying to find a geometric counterpart of it. More details are in {Section \ref{sec:Lagcov}}.}

\section{Preliminary}\label{sec:p}
This section provides an introduction to basic notations from singularity theory and Lagrangian Floer theory.

\subsection{Jet bundles and singularities}
This subsection contains the definition of jet bundles. Using the techniques of jet bundles provided in Golubitsky and Guillemin's book \cite{golubitsky2012stable}, it can be shown that the singularities of smooth maps between closed surfaces are stably one dimensional submanifolds whose elements are fold points and finitely many isolated cusp points, that is to say, the singularities of nearby smooth maps are also one dimensional submanifolds whose elements are fold points and finitely many isolated cusp points with the same number. This subsection is necessary for Section \ref{sec:immerofbisin}. For details, see \cite{golubitsky2012stable}.

Here is the outline of this subsection. We first give the definition of first jet bundles. Then Whitney's fold singularity theorem (Theorem \ref{thm:g}) is stated. After that, the definition of second jet bundles is presented. Finally, we introduce Whitney's cusp singularity theorem (Theorem \ref{thm:2ordtran}).\\

Let $X$ and $Y$ be smooth surfaces. Suppose $\pi_1$ and $\pi_2$ are the projections from $X\times Y$ to $X$ and $Y$, respectively. The {\bf first jet bundle} $\mathbf{J^1(X,Y)}$ is defined as 
\begin{equation}\label{equ:2.2}
J^1(X,Y):=\pi_1^*(T^*X)\otimes\pi_2^*(TY)(\cong\Hom(\pi_1^*(TX),\pi_2^*(TY))).
\end{equation}
This is a bundle over $X\times Y$ whose fiber over $(x,y)\in X\times Y$ is $\Hom(T_xX,T_yY)$.\\
Notice that the space of rank one $2\times2$ matrices can be written as{
\begin{equation*}
    \begin{bmatrix}
    a & ca\\
    b & cb
    \end{bmatrix}\ \ \mathrm{or}\ \ 
    \begin{bmatrix}
    ca & a\\
    cb & b
    \end{bmatrix}\ \ a,b,c\in\R.
\end{equation*}
Therefore this space is the union of two subspaces of $\Hom(\R^2,\R^2)$ with dimension three.}

\begin{defn}\label{def:s1}
Let $\mathbf{S_1(X,Y)}$ be the subbundle of $J^1(X,Y)\rightarrow X\times Y$ whose fiber over a point $(x,y)\in X\times Y$ is the submanifold of rank one linear maps in $\Hom(T_xX,T_yY)$. To simplify the notation,  we use $S_1$ to represent $S_1(X,Y)$.
\end{defn}

\begin{defn}
Suppose that $f:X\rightarrow Y$ is a smooth map between two manifolds. Given a submanifold $A\subset Y$, we say that {\bf $\mathbf f$ is transverse to $\mathbf A$ at $\mathbf x$} if 
\[
\im(df_x)+T_{f(x)}A=T_{f(x)}Y.
\]
Suppose  $f$ is transverse to $A$ at every $x\in f^{-1}(A)$, then $f$ is said to {\bf be transverse to $\mathbf A$}.
\end{defn}

One of the most important theorems is the following.

\begin{them}[Transversality Theorem, Theorem 4.4 \cite{golubitsky2012stable}]
Let $f:X\rightarrow Y$ be a smooth map between two manifolds. Suppose $A\subset Y$ is a submanifold. If $f$ is transverse to $A$, then $f^{-1}(A)$ is a smooth submanifold of $X$. In particular, if $A=\{y\}$ is a single point and $df_x$ is surjective for all $x\in f^{-1}(y)$, then $f^{-1}(y)$ is a smooth submanifold of $X$.
\end{them}

\begin{rmk}
Let $f:X\rightarrow Y$ be a smooth map between two surfaces. The tangent map $df$ of $f$ is a map from $X$ to $J^1(X,Y)$. If $df$ is transverse to $S_1$, then the transversality theorem asserts that $S_1(f):=(df)^{-1}(S_1)$ is a smooth submanifold of $X$. Call $S_1(f)$ the {\bf singular set} of $f$.
\end{rmk}

The proof of the following theorem can be found in \cite{golubitsky2012stable}.

\begin{them}[Whitney]\label{thm:g}
Suppose that $f:X\rightarrow Y$ is a smooth map between surfaces. Assume $df$ is transverse to $S_1$ at $x$. If the restriction of $f$ to $S_1(f)$ is an immersion, then $x$ is a fold point of $f$.
\end{them}


The next task is to define the second jet bundle.

Suppose $X$ and $Y$ are two smooth surfaces. Denote $\pi_1$ and $\pi_2$ as the projections from $X\times Y$ to $X$ and $Y$, respectively. Let
\[
E=\pi_1^*(TX),F=\pi_2^*(TY).
\]
The {\bf the second jet bundle} $\mathbf{J^2(X,Y)}$ is defined as
\begin{equation*}
    J^2(X,Y):=\Hom(E,F)\oplus \Hom(\Sym^2E,F),
\end{equation*}
where the symmetric product is{
\[
\Sym^2 E:=\Span\{\frac12v_1\otimes v_2+\frac12v_2\otimes v_1|v_1\otimes v_2\in E\otimes E\}.
\]}
By Equation (\ref{equ:2.2})
\[
J^1(X,Y)\cong\Hom(E,F),
\]
so there is a canonical projection map from $J^2(X,Y)$ to $J^1(X,Y)$. \\
\begin{rmk}
Let $f:X\rightarrow Y$ be a smooth map between surfaces. Then $f$ induces a smooth map (defined later)
\[
j^2(f):X\rightarrow J^2(X,Y).
\]
\end{rmk}

Recall that $S_1$ is a submanifold in $J^1(X,Y)$ defined independently of $f$ (Definition \ref{def:s1}). The singular set $S_1(f)$ is the preimage of $S_1$ under $df$. Let 
\[
S_{1,1}(f):=\{x|x\in S_1(f), T_xS_1(f)=\ker(df)_x\}.
\]
In generic case, Theorem \ref{thm:2ordtran} shows that $S_{1,1}(f)$ corresponds to cusp points of $f$.

The following is to define a submanifold (denoted as $S_{1,1}$) of $J^2(X,Y)$ whose preimage under $j^2(f)$ is $S_{1,1}(f)$.\\
Define 
\begin{align*}
    \pi: E\otimes E&\rightarrow\ \ \ \ \ \ \ \Sym^2 E\\
    v_1\otimes v_2&\mapsto \frac12v_1\otimes v_2+\frac12v_2\otimes v_1.
\end{align*}
This induces a map over the bundles
\begin{align*}
    \pi^*_E:\Hom(\Sym^2E,F)&\rightarrow  \Hom(E\otimes E,F)\\
    f\ \ \ \ \ \ \ \ \ \ \ &\mapsto\ \ \ \ \ \ f\circ\pi
\end{align*}
Since 
\[
\Hom(E\otimes E,F)\cong \Hom(E,\Hom(E,F)),
\]
we have the following compositions
\begin{equation}\label{equ:2.4}
    \bar\pi_E:\Hom(\Sym^2E,F)\stackrel{\pi^*_E}\longrightarrow \Hom(E\otimes E,F)\cong \Hom(E,\Hom(E,F)).
\end{equation}
Notice that the image of the map $\pi^*_E$ are linear maps $g:E\otimes E\rightarrow F$ symmetric in $E\otimes E$ and
\[
\pi^*_E:\Hom(\Sym^2E,F)\rightarrow \im\pi^*_E
\]
is an isomorphism. Suppose $\tilde g\in \Hom(E,\Hom(E,F))$ is symmetric in $E\otimes E$ after applying the isomorphism in Equation (\ref{equ:2.4}), then we can regard $\tilde g$ as an element in $\Hom(\Sym^2E,F)$.

\begin{rmk}\label{rmk:loccoord}
We diverge here to give the definition of $j^2(f)$ using Equation (\ref{equ:2.4}), where $f:X\rightarrow Y$ is a smooth map between surfaces. From the last section, $df$ can be regarded as the map
\begin{align*}
    df:X&\rightarrow \Hom(E,F)\\
    x&\mapsto (x,f(x),df_x).
\end{align*}
Therefore take the differential once again to get
\begin{align*}
    d(df):TX&\rightarrow\ \ \ \ \ \ \ \ \ \ \ T\Hom(E,F)\\
    (x,v)\ \ &\mapsto (x,f(x),df_x,dx(v),df_x(v),d(df_x)(v)).   
\end{align*}
Notice that the last coordinate is in $T(\Hom(E,F)|_{(x,f(x))})\cong\Hom(E,F)|_{(x,f(x))}$, thus
\[
d(df_x)\in\Hom(T_xX,\Hom(E,F)|_{(x,f(x))}).
\]
Since $T_xX$ is identified with $E|_{(x,f(x))}$, then
\[
d(df_x)\in\Hom(E|_{(x,f(x))},\Hom(E,F)|_{(x,f(x))})\cong\Hom(E|_{(x,f(x))}\otimes E|_{(x,f(x))},F|_{(x,f(x))}).
\]
Locally,
\[
d(df_x)(\frac{\partial}{\partial x_i},\frac{\partial}{\partial x_j})=d(df_x)(\frac{\partial}{\partial x_j},\frac{\partial}{\partial x_i})=\frac{\partial^2 f}{\partial x_ix_j},\ \ i,j=1,2.
\]
This implies that $d(df_x)(\cdot,\cdot)$ is symmetric. As a result, compare with the maps in Equation (\ref{equ:2.4}),
\[
d(df_x)\in\Hom(\Sym^2E,F)|_{(x,f(x))}.
\]
So $j^2(f)$ can be defined as
\begin{align*}
    j^2(f):X&\rightarrow\ \ \ \ \ J^2(X,Y)\\
    x&\mapsto (x,f(x),df_x,d(df_x)).
\end{align*}
\end{rmk}

{\bf Then we continue to define} $\mathbf{S_{1,1}}$. Let $(x,y,\alpha)\in S_1$. By the definition of $S_1$, 
\[
\alpha\in \Hom(T_xX,T_yY) \ \text{and}\ \rank(\alpha)=1.
\]
As a result,
\[
K_{\alpha}:=\ker\alpha\subset T_xX\ \text{and}\ L_{\alpha}:=\coker\alpha\subset T_yY
\]
are one dimensional vector spaces. Let {
\[
\Sym^2 K_{\alpha}:=\Span\{\frac12v_1\otimes v_2+\frac12v_2\otimes v_1|v_1\otimes v_2\in K_{\alpha}\otimes K_{\alpha}\}.
\]}
Then we have a bundle (\cite{golubitsky2012stable} Section 4, Chapter 6) over $S_1$
\[
\Hom(\Sym^2K,L):=\cup_\alpha\Hom(\Sym^2K_\alpha,L_\alpha)\rightarrow S_1.
\]
Define 
\begin{align*}
    \pi_{K_\alpha}: K_{\alpha}\otimes K_{\alpha}&\rightarrow\ \ \ \ \ \Sym^2 K_{\alpha}\\
    v_1\otimes v_2\ \ \ &\mapsto \frac12v_1\otimes v_2+\frac12v_2\otimes v_1.
\end{align*}
This induces a map over the bundles
\begin{align*}
    \pi^*_{K}:\Hom(\Sym^2K,L)&\rightarrow \Hom(K\otimes K,L)\\
    f\ \ \ \ \ \ \ \ \ &\mapsto\ \ \ \ \ \ f\circ\pi
\end{align*}
Since 
\[
\Hom(K\otimes K,L)\cong \Hom(K,\Hom(K,L)),
\]
one can get the following compositions
\[
\bar\pi_K:\Hom(\Sym^2K,L)\stackrel{\pi^*_K}\longrightarrow \Hom(K\otimes K,L)\cong \Hom(K,\Hom(K,L)).
\]
Combine this with Equation (\ref{equ:2.4}), the following diagram is commutative
\begin{equation*}
    \xymatrix{
    \Hom(\Sym^2K,L)\ar[r]&\Hom(K\otimes K,L)\ar[r]^{\cong\ \ \ } &\Hom(K,\Hom(K,L))\\
    \Hom(\Sym^2E,F)|_{S_1}\ar[u]^\eta\ar[r]&\Hom(E\otimes E,F)|_{S_1}\ar[u]^\eta\ar[r]^{\cong\ \ \ }&\Hom(E,\Hom(E,F))|_{S_1}\ar[u]^\eta,
    }
\end{equation*}
where the vertical maps $\eta$ are given by first restricting the corresponding maps to $K$ and then projecting to $L$. Equipped with the above diagram, define
\[
S_{1,1}(X,Y):=\{(x,y,\alpha,\beta)\in J^1(X,Y)\oplus\Hom(\Sym^2E,F)|(x,y,\alpha)\in S_1,\ker\bar\pi_K(\eta(\beta))=\ker\alpha\}.
\]
If no confusion is caused, then $S_{1,1}$ is used instead of $S_{1,1}(X,Y)$.

The next theorem is about the properties of $S_{1,1}$.

\begin{them}[\cite{ronga1971calcul} Proposition 3.2, \cite{golubitsky2012stable} Theorem 4.7 Chapter 6]\label{thm:2.8}
Let $f:X\rightarrow Y$ be a smooth map between surfaces. Suppose $df$ is transverse to $S_1$. Then 
\[
\{x|x\in S_{1,1}(f)\}=\{x| j^2(f)(x)\in S_{1,1}\}.
\]
\end{them}

\begin{rmk}\label{rmk:2.2}
According to \cite{ando1982elimination}, $S_{1,1}$ is a codimension two submanifold of $J^2(X,Y)$. We illustrate this claim by calculating in local coordinates. (This is needed in Theorem \ref{thm:locim}.) Let $f=(f_1,f_2):X\rightarrow Y$ be a smooth map between surfaces. Then $f$ induces the map
\[
j^2(f):X\rightarrow J^2(X,Y).
\]
Recall Remark \ref{rmk:loccoord} shows that in local coordinates
\[
j^2(f)(x)=(x,f(x),df_x,\frac{\partial^2 f}{\partial x_1^2}(x),\frac{\partial^2 f}{\partial x_1x_2}(x),\frac{\partial^2 f}{\partial x_2^2}(x)).
\]
Suppose $x\in S_{1,1}(f)$. Since $S_1(f)$ is a submanifold of $X$ and $T_xS_1(f)\subset \ker df_x$, the coordinates around $x$ and $f(x)$ can be chosen such that
\begin{equation*}
    S_1(f)=\{(x_1,x_2)|x_1=0\},\ \ 
    df_x=
    \begin{bmatrix}
    1&0\\
    0&0
    \end{bmatrix}.
\end{equation*}
So $\frac{\partial}{\partial x_2}\in\ker df_x$. According to the previous theorem, 
\[
x\in S_{1,1}(f)\Longleftrightarrow j^2(f)(x)\in S_{1,1}(f).
\]
As a result, by the definition of $S_{1,1}$, the map
\begin{align*}
    \bar\pi_K\circ\eta(d(df_{x}))(\frac{\partial}{\partial x_2}):\ker df_x&\rightarrow \coker df_x\\
    \frac{\partial}{\partial x_2}\ \ \ \ \ \ \ \ \ \ \ &\mapsto\ \  \frac{\partial^2 f}{\partial x_2^2}(x)
\end{align*}
has to be 0. This implies that 
\[
\frac{\partial^2 f}{\partial x_2^2}(x)=0.
\]
Therefore
\[
j^2(f)(x)=(x,f(x),df_x,\frac{\partial^2 f}{\partial x_1^2}(x),\frac{\partial^2 f}{\partial x_1x_2}(x),0).
\]
This shows that the last component in the above formula gives one of the codimensions of $S_{1,1}$ in $J^2(X,Y)$. For the remaining one codimension, notice that the first three coordinates $(x,f(x),df_x)$ correspond to the components in $S_1$. Since $S_1$ has codimension one in $J^1(X,Y)$ and $J^2(X,Y)=J^1(X,Y)\oplus\Hom(\Sym^2E,F)$, then the remaining one codimension of $S_{1,1}$ in $J^2(X,Y)$ corresponds to the codimension of $S_1$ in $J^1(X,Y)$.
\end{rmk}

\begin{them}[Whitney, cf. Golubitsky and Guillemin \cite{golubitsky2012stable} Section 4, Chapter 6]\label{thm:2ordtran}
Let $f:X\rightarrow Y$ be a smooth map between surfaces. Suppose $df$ and $j^2(f)$ intersect with $S_1$ and $S_{1,1}$ transversely, respectively. Then 
\begin{itemize}
    \item $S_{1,1}(f)$ is a zero dimensional submanifold of $S_1(f)$,
    \item there are local coordinates around $x\in S_{1,1}(f)$ and $f(x)$ such that $f$ locally can be written as
    \[
    f(x_1,x_2)=(x_1,x_1x_2+x_2^3).
    \]
    In particular, $x$ is a {\bf cusp point}.
\end{itemize}
\end{them}

\begin{rmk}
The following calculation shows that all points near a cusp point are fold points. Let
\begin{align*}
    f:\R^2&\rightarrow\ \ \ \ \ \ \R^2\\
    (x_1,x_2)&\mapsto(x_1,x_2+x_2^3).
\end{align*}
The Jacobian of $f$ is
\begin{equation*}
    df=
    \begin{bmatrix}
    1&0\\
    x_2&x_1+3x_2^2
    \end{bmatrix}.
\end{equation*}
Thus $(x_1,x_2)$ is in the singular set iff $x_1=-3x_2^2$. If the singular set is an immersion at $(x_1,x_2)$ under $f$, then $(x_1,x_2)$ is a fold point according to Whitney's theorem. This is equivalent to
\begin{equation*}
    \begin{bmatrix}
    0\\
    0
    \end{bmatrix}
    \ne
    \begin{bmatrix}
    1&0\\
    x_2&0
    \end{bmatrix}
    \begin{bmatrix}
    -6x_2\\
    1
    \end{bmatrix}
    =
    \begin{bmatrix}
    -6x_2\\
    -6x_2^2
    \end{bmatrix}.
\end{equation*}
The above holds if $x_2\ne0$. Therefore the only cusp singular point is $(0,0)$.
\end{rmk}


\subsection{An introduction to symplectic topology}
This subsection includes basic definitions of symplectic topology, the Darboux theorem, and the Weinstein tubular neighbourhood theorem. More details can be found in McDuff and
Salamon \cite{mcduff2017introduction}.

\begin{defn}
Let $X$ be a smooth manifold with $2n$ dimension. A 2-form $\omega$ is called a {\bf symplectic form} if the following conditions hold:
\begin{itemize}
    \item $\omega$ is a closed form.
    \item The restriction of $\omega$ to each tangent space of $X$ is skew-symmetric and non-degenerate.
\end{itemize}
The pair $(X,\omega)$ is used to represent a symplectic manifold $X$ with the symplectic form $\omega$. When no problem is caused, we just say that $X$ is a symplectic manifold without introducing its symplectic form.
\end{defn}

\begin{ex}
The real vector space $\R^{2n}$ is a symplectic manifold with its symplectic form defined as 
\[
\omega_{std}=\sum_{i=1}^ndx_i\wedge dy_i,
\]
where the coordinate of $\R^{2n}$ is given by $(x_1,\ldots,x_n,y_1\ldots,y_n)$.
\end{ex}

\begin{ex}\label{exp:surface}
Any {oriented} closed surface $F$ admits a symplectic structure. The volume form for any Riemannian metric is a sympletic form.
\end{ex}

\begin{ex}
Suppose that $(X_1,\omega_1)$ and $(X_2,\omega_2)$ are two symplectic manifolds. Then $(X_1\times X_2, w_1\times(-\omega_2))$ is a symplectic manifold.
\end{ex}

\begin{ex}
The cotangent bundle $T^*X$ of any smooth manifold $X$ is a symplectic manifold. 
\end{ex}

\begin{defn}
Let $(X_1,\omega_1)$ and $(X_2,\omega_2)$ be two symmplectic manifolds. Suppose $f:X_1\rightarrow X_2$
is a smooth map.
The function $f$ is a {\bf (local) symplectomorphism} if $f$ is a $($local$)$ diffeomophism and
\[
f^*\omega_2=\omega_1.
\]
In this case, $X_1$ and $X_2$ are {\bf (local) symplectomorphic}.
\end{defn}

\begin{ex}\label{exp:Hamflow}
This example gives a family of symplectomorphisms called the {\bf Hamiltonian flow}. Let $(\R^{2n},\omega_{std})$ be the standard symplectic manifold over $\R^{2n}$ with
\[
\omega_{std}=\sum_{i=1}^{n}dx_i\wedge dy_i.
\]
Suppose that
\begin{align*}
    h:\R^{2n}&\rightarrow \R\\
    (x_1,\ldots,x_n,y_1,\ldots,y_n)&\mapsto h(x_1,\ldots,x_n,y_1,\ldots,y_n)
\end{align*}
is a smooth {compact} supported function. Define the {\bf Hamiltonian vector field} $X$ of $h$ as
\[
X=\sum_{i=1}^n\frac{\partial h}{\partial y_i}\frac{\partial}{\partial x_i}-\sum_{i=1}^n\frac{\partial h}{\partial x_i}\frac{\partial}{\partial y_i}.
\]
Note that
\[
\iota_X\omega_{std}=dh,
\]
where $\iota_X$ is the contraction
\[
\iota_X\omega_{std}(\cdot)=\omega_{std}(X,\cdot).
\]
According to the theory of O.D.E (\cite{arnold1992ordinary} Section 31), the solution $\phi_t$ to the following system of equations uniquely exists
\begin{equation*}
    \left\{
    \begin{aligned}
    \dot{x}&=\ \ \frac{\partial h}{\partial y_i},\quad i=1,\ldots,n\\
    \dot{y}&=-\frac{\partial h}{\partial x_i},\quad i=1,\ldots,n
    \end{aligned}
    \right..
\end{equation*}
The solution $\phi_t$ is called the {\bf Hamiltonian flow} of $h$. Since the flows generated by vector fields are diffeomorphisms, if $\phi_t$ also preserves the symplectic form $\omega_{std}$, then $\phi_t$ is a symplectomorphism for all $t$. In fact, 
\begin{equation*}
    \begin{aligned}
    &\frac{d}{dt}\phi_t^*\omega_{std}\\
    =&d\iota_X\omega_{std}+\iota_Xd\omega_{std}\\
    =&ddh\\
    =&0.
    \end{aligned}
\end{equation*}
The first equality holds because of Cartan's formula (\cite{mcduff2017introduction} Section 3.1).
Therefore 
\[
\phi_t^*\omega_{std}=\omega_{std},\quad \forall t.
\]
This shows that the Hamiltonian flow of a smooth function is a symplectomorphism.
\end{ex}

\begin{rmk}
The Hamiltonian vector field $X$ of $h$ vanishes outside the support of $h$. Therefore the Hamiltonian flow $\phi_t$ is trivial outside the support of $h$.
\end{rmk}

\begin{defn}
Let $X$ and $Y$ be smooth manifolds. The triple $(X,f;Y)$ is defined as an {\bf immersed submanifold} of $Y$ with an immersion map
\[
f:X\rightarrow Y,
\]
if $f$ is smooth and its tangent map $df_x$ is injective for each point $x\in X$. For simplicity, denote by $X\looparrowright Y$ an immersion when the immersion map is not specified.
\end{defn}

\begin{defn}
Suppose that $(L,l;X)$ is an immersion from an $n$ dimensional manifold $L$ to a 2n dimensional symplectic manifold $(X,\omega)$. The immersion $(L,l;X)$ is called  a {\bf Lagrangian immersion} if
\[
l^*\omega=0.
\]
\end{defn}

\begin{ex}
Consider $(\R^{2n},\omega_{std})$. Let
\begin{equation}\label{lagbas}
\begin{split}
    \varepsilon_1,\varepsilon_2,\ldots,\varepsilon_n\\
    \epsilon_1,\epsilon_2,\ldots,\epsilon_n
\end{split}    
\end{equation}
be an orthonormal basis of $\R^{2n}$, where $\varepsilon_i$ stands for the vector with i-th entry being 1 and the rest being 0, and $\epsilon_j$ stands for the vector with $($j+n$)$-th entry being 1 and the rest being 0. Pick one vector from each pair {$\{\varepsilon_i,\epsilon_i\}$}, $i=1,\ldots,n,$ to get $n$ vectors. The space spanned by these $n$ vectors is a Lagrangian subspace.
\end{ex}

\begin{ex}\label{exp:2.1}
Any immersed curve $C$ in a symplectic surface $(F,\omega)$ is Lagrangian. Since $\omega$ is skew-symmetric, then $\omega(v_x,v_x)=0$ for all $x\in C$ and $v_x\in T_xC$.
\end{ex}

\begin{defn}\label{def:lagint}
Let $(L_1,l_1;X)$ and $(L_2,l_2;X)$ be two Lagrangian immersions into a symplectic manifold $X$. The {\bf fiber product} $L_1\times_X L_2$ is defined as 
\[
L_1\times_X L_2=\{(x_1,x_2)\in L_1\times L_2|\ l_1(x_1)=l_2(x_2)\}.
\]
\end{defn}

\begin{defn}
{Let $(L_1,l^1;X)$ and $(L_2,l^2;X)$ be two Lagrangian immersions into a symplectic manifold $X$. Let $(x_1,x_2)\in L_1\times_XL_2$ be an intersection point of these two Lagrangian immersions. We say {\bf $\mathbf{L_1}$ intersects $\mathbf{L_2}$ transversely at $\mathbf{(x_1,x_2)}$} if 
\[
\mathrm{Im}(df^1)_{x_1}\cap\mathrm{Im}(df^2)_{x_2}=\{0\}.
\]
If the above equation holds for every intersection point, then {\bf $\mathbf{L_1}$ intersects $\mathbf{L_2}$ transversely}.}
\end{defn}

\begin{rmk}\label{rmk:lagint}
The terminology {\bf fiber product} comes from category theory. The intersections of immersed manifolds $(L_1,l_1;X)$ and $(L_2,l_2;X)$ in Definition \ref{def:lagint} cannot be defined simply as $l_1(L_1)\cap l_2(L_2)$. Since $(L_1,l_1;X)$ and $(L_2,l_2;X)$ are immersions, $l_i^{-1}(x)$ may contain more than one element for $x\in l_1(L_1)\cap l_2(L_2)$ for $i=1,2$. It is necessary to distinguish these preimages. So we use fiber products instead of intersections here. {To stress on the geometric aspect, the elements in the fiber product are still called intersections}.
\end{rmk}



The following two theorems are fundamental in symplectic topology .
\begin{them}[Darboux Theorem]{\rm\cite{mcduff2017introduction}}
Let $(X^{2n},\omega)$ be a symplectic manifold. For any point $x\in X$, there is an open neighbourhood $U\subset X$ of $x$ such that $(U,\omega|_U)$ is symplectomorphic to $(\R^{2n},\omega_{std})$.
\end{them}

\begin{them}[Weinstein Tubular Neighbourhood Theorem]\label{thm:tubu}{\rm \cite{eliashberg2002introduction}}
Let $(X,\omega)$ be a symplectic manifold. Assume that $f:L\rightarrow X$ is a Lagrangian immersion. Then there is a local symplectomorphism $G$ from a tubular neighbourhood $T^*_\varepsilon L$ of the zero section of the cotangent bundle $T^*L$ to a neighbourhood $U\supset f(L)$ as in the following commutative diagram:
\begin{equation*}
    \xymatrix{
    T_\varepsilon^*L \ar[dr]^G & \\
    L \ar[r]^{f\ \ \ } \ar[u] &U\subset X}.
\end{equation*}
\end{them}

\begin{rmk}
The local symplectic diffeomorphism $G$ need not to be an injective map.
\end{rmk}
{
\begin{rmk}
    Suppose $X$ and $L$ in the above theorem are open manifolds. The same argument can be shown to get that suppose $D\subset\bar D\subset L$ is an open disc, then there is a local symplectomorophism $G$ from a tubular neighbourhood $T^*_\varepsilon L|_D$ of the zero section of the cotangent bundle $T^*L|_D$ to a neighbourhood $U\supset f(D)$ as in the following commutative diagram:
\begin{equation*}
    \xymatrix{
    T_\varepsilon^*L|_D \ar[dr]^G & \\
    D \ar[r]^{f|_D\ \ \ } \ar[u] &U\subset X}.
\end{equation*}
\end{rmk}}

\section{Singularities of Lagrangian immersions}\label{sec:singular}
Let $F_1$ and $F_2$ be two closed symplectic surfaces. Assume that $g=(g_1,g_2):F\rightarrow F_1\times F_2$ is a Lagrangian immersion. This section shows that the critical points of $g_1$ and $g_2$ coincide.



Let {$f:X\rightarrow Y$} be a smooth map between two Riemannian manifolds with the same dimension. Suppose $U\subset X$ and $V\subset Y$ are open neighbourhoods of the local charts. Then one can calculate the determinant of $df|_U$ at every point in $U$. Assume $\varphi:U\rightarrow U',$ and $\psi:V\rightarrow V'$ are transition functions of {Riemannian manifolds} $X$ and $Y$ respectively, {then the Jacobian of $\varphi$ and $\psi$ are equal to 1. So the following is true:}
\begin{equation*}
    \det(d(\varphi\circ f\circ\psi^{-1}))=\det(d\varphi)\det(df)\det(d\psi^{-1})=\det(df).
\end{equation*}
Therefore the determinant map $\det(df)$ of smooth functions between Riemannian manifolds is well defined and $\det(df)$ is defined as the {\bf determinant map} of $f$.\\

{Let $g=(g_1,g_2):F\rightarrow F_1\times F_2$ be a Lagrangian immersion. According to Definition \ref{def:bisin}, $x\in F$ is a bisingular point iff $x$ is a critical point for both $g_1$ and $g_2$, or equivalently, $\det(dg_1)(x)=\det(dg_2)(x)=0$.}

{
\begin{rmk}
    Every closed smooth manifold can be equipped with a Riemannian metric.
\end{rmk}
}

\begin{lem}\label{lem:det}
Let $(F_1,\omega_1)$ and $(F_2,\omega_2)$ be two closed symplectic surfaces. Equip $F_1\times F_2$ with the symplectic form $\omega_1\times(-\omega_2)$. Suppose that 
\[
g=(g_1,g_2):F\rightarrow F_1\times F_2
\]
is a Lagrangian immersion. {Then with fixed Riemannian metrics on $F_1$ and $F_2$,}
\[
\det(dg_1)=\det(dg_2).
\]
\end{lem}

\begin{proof}
Let 
\[
dg: TF\rightarrow T(F_1\times F_2)
\]
be the tangent map of $g$. Take a point $(x_1,x_2)\in F$. By Darboux's theorem, there are local coordinates such that the symplectic form $\omega_i$ around $g_i(x_1,x_2)\in F_i$ is the standard one on $\mathbb R^2$, for $i=1,2$. Since $g$ is a Lagrangian immersion, the following is true under these coordinates
\begin{equation*}
    \begin{split}
        0&=g^*(\omega_1\times(-\omega_2))(\frac{\partial}{\partial x_1},\frac{\partial}{\partial x_2})\\
        &=g_1^*(\omega_1)(\frac{\partial}{\partial x_1},\frac{\partial}{\partial x_2})-g_2^*(\omega_2)(\frac{\partial}{\partial x_1},\frac{\partial}{\partial x_2})\\
        &=\det(dg_1)-\det(dg_2).
    \end{split}
\end{equation*}
\end{proof}

\begin{cor}\label{cor:det}
Let $(F_1,\omega_1)$ and $(F_2,\omega_2)$ be two closed symplectic surfaces. Given a Lagrangian immersion 
\[
g=(g_1,g_2):F\rightarrow F_1\times F_2,
\]
then the critical sets of $g_1$ and $g_2$ are the same. Since {
\[
1\le\rank(dg_1)_x\le2,1\le\rank(dg_2)_x\le2,\forall x\in F,
\]
then with fixed Riemannian metrics on $F,F_1,F_2$,}
\[
\det(dg_1)(x)=\det(dg_2)(x)=0\Rightarrow \rank(dg_1)_x=\rank(dg_2)_x=1.
\]
\end{cor}

\section{Transversality of 1-jet}

This and the next sections prove the main theorem (Theorem \ref{thm:main}) of the first part. Since its proof is quite long, the idea is explained here. The strategy is to apply Whitney's theorems (Theorem \ref{thm:g}, Theorem \ref{thm:2ordtran}). Let $(F_1,\omega_1)$ and $(F_2,\omega_2)$ be two closed symplectic surfaces. Suppose that
\[
g=(g_1,g_2):F\rightarrow F_1\times F_2
\]
is a smooth Lagrangian immersion from a closed surface to $(F_1\times F_2,\omega_1\times(-\omega_2))$. According to Corollary \ref{cor:det}, since $g$ is a Lagrangian immersion, then
\[
\rank(dg_1)=\rank(dg_2)\ne0.
\]
Whitney's theorems indicate that, if
\begin{itemize}
    \item the tangent map $dg_i$ is transverse to $S_1(F,F_i)$ (Definition \ref{def:s1}) for $i=1,2$,
    \item after removing finitely many points, the singular set $S_1(g_i)$ is immersed into $F_i$ under $g_i$, for $i=1,2$,
    \item $j^2(g_i)$ is transverse to $S_{1,1}(F,F_i)$ at the finitely many points in the last entry, for $i=1,2$,
\end{itemize}
then $g_i$ only has fold singularities and finitely many cusp singularities.

We construct a perturbation in this section such that the first condition is satisfied. Another perturbation is constructed in the next section such that the second and the third conditions hold. The existence of both perturbations is proved locally first (Theorem \ref{thm1} and Theorem \ref{thm:locim}). A global perturbation (Theorem \ref{thm2} and Theorem \ref{main5}) is constructed using a partition of unity and the local perturbation.

\subsection{Local solution}
We first work on the local case. Let $(\R^2,g=(g_1,g_2);\R^2\times\R^2)$ be a Lagrangian immersion into $(\mathbb R^2\times \mathbb R^2,\omega_{std}\times(-\omega_{std}))$. Suppose that $dg_i$ is the tangent map of $g_i$ for $i=1,2$. Then Corollary \ref{cor:det} implies that
\[
 \rank(dg_1)_x=\rank(dg_2)_x\ge1,\forall x\in\R^2.
\]
Let $S_1$ be the subbundle of the first jet bundle $J^1(\R^2,\R^2)$ whose fiber consists of rank one matrices. Note that $dg_i$ can also be viewed as a map from $\R^2$ to $J^1(\R^2,\R^2)$. Since  the rank of $g_1$ cannot be zero, $g_1$ has critical points if and only if the image of $dg_1$ intersects $S_1$. In this section, we prove that after performing a perturbation, $dg_1$ intersects $S_1$ transversely. 

Because $\det(dg_1)=\det(dg_2)$, according to the lemma below, if $dg_1$ is transverse to $S_1$, then $dg_2$ is also transverse to $S_1$.

\begin{lem}\label{lem:trancri}
Suppose that
\[
f=(f_1,f_2):\R^2\rightarrow\R^2
\]
is a smooth map such that at every point
\[
\rank(df)\ge1.
\]
Then $df$ is transverse to $S_1$ in $J^1(\R^2,\R^2)$ if and only if $0$ is a regular value of $\det(df)$.
\end{lem}

\begin{proof}
This lemma only needs to be proved locally. {\bf The first step} is to construct suitable local coordinates such that we are able to perform calculations. Assume $df$ intersects $S_1$ at $(x_0,df_{x_0})$. Let $(x_1,x_2)$ be a local coordinate of $\R^2$ around ${x_0}$. Let 
\begin{equation*}
    (x_1,x_2,\left[ 
    \begin{aligned}
    a\ \ \ b\\
    c\ \ \ d
    \end{aligned}
    \right])
\end{equation*}
be a local coordinate of $J^1(\R^2,\R^2)$ around $({x_0},df_{x_0})$. Then under theses local coordinates,
\begin{equation*}
    (x_1,x_2,\left[ 
    \begin{aligned}
    a\ \ \ b\\
    c\ \ \ d
    \end{aligned}
    \right])|_{({x_0},df_{x_0})}=(x^0_1,x^0_2,\left[
    \begin{aligned}
    a_0\ \ \ b_0\\
    c_0\ \ \ d_0
    \end{aligned}
    \right])=({x_0},\left[ 
    \begin{aligned}
    \frac{\partial f_1}{\partial x_1}\ \ \ \frac{\partial f_1}{\partial x_2}\\
    \frac{\partial f_2}{\partial x_1}\ \ \ \frac{\partial f_2}{\partial x_2}
    \end{aligned}
    \right]({x_0})).
\end{equation*}
Since $({x_0},df_{x_0})\in S_1$, then $( \frac{\partial f_1}{\partial x_1},\frac{\partial f_2}{\partial x_1})^T$ and $(\frac{\partial f_1}{\partial x_2},\frac{\partial f_2}{\partial x_2})^T$ are linear dependent at $x$. Therefore, there is a number $\lambda_0$ such that either
\begin{equation}\label{equ:4.3}
    \left[ 
    \begin{aligned}
    \frac{\partial f_1}{\partial x_1}\ \ \ \frac{\partial f_1}{\partial x_2}\\
    \frac{\partial f_2}{\partial x_1}\ \ \ \frac{\partial f_2}{\partial x_2}
    \end{aligned}
    \right]({x_0})=
    \left[ 
    \begin{aligned}
    \lambda_0\frac{\partial f_1}{\partial x_2}\ \ \ \frac{\partial f_1}{\partial x_2}\\
    \lambda_0\frac{\partial f_2}{\partial x_2}\ \ \ \frac{\partial f_2}{\partial x_2}
    \end{aligned}
    \right]({x_0}):=\left[ 
    \begin{aligned}
    \lambda_0 b_0\ \ \ b_0\\
    \lambda_0 d_0\ \ \ d_0
    \end{aligned}
    \right]
\end{equation}
{or
\begin{equation}\label{equ:4.3.1}
    \left[ 
    \begin{aligned}
    \frac{\partial f_1}{\partial x_1}\ \ \ \frac{\partial f_1}{\partial x_2}\\
    \frac{\partial f_2}{\partial x_1}\ \ \ \frac{\partial f_2}{\partial x_2}
    \end{aligned}
    \right]({x_0})=
    \left[ 
    \begin{aligned}
    \frac{\partial f_1}{\partial x_2}\ \ \ \lambda_0\frac{\partial f_1}{\partial x_2}\\
    \frac{\partial f_2}{\partial x_2}\ \ \ \lambda_0\frac{\partial f_2}{\partial x_2}
    \end{aligned}
    \right]({x_0}):=\left[ 
    \begin{aligned}
     b_0\ \ \ \lambda_0b_0\\
     d_0\ \ \ \lambda_0d_0
    \end{aligned}
    \right]
\end{equation}
at $({x_0},df_{x_0})$. We first deal with the case in Equation (\ref{equ:4.3}).} Because the fiber of $S_1$ at ${x_0}$ are matrices of rank one, without loss of generality, assume 
\begin{equation*}
    (x_1,x_2,
    \left[ 
    \begin{aligned}
    \lambda b\ \ \ b\\
    \lambda d\ \ \ d
    \end{aligned}
    \right])
\end{equation*}
is the local coordinate of $S_1$ around $({x_0},df_{x_0})$. By counting the number of variables, $S_1$ is a codimension one submanifold in $J^1(\R^2,\R^2)$.\\
{\bf The second step} is to prove the lemma under the coordinates constructed in the first step. In these coordinates, the tangent space of $S_1$ is generated by
\begin{equation}\label{equ:4.1}
    (\frac{\partial}{\partial x_1},0,\left[ 
    \begin{aligned}
    0\ \ \ 0\\
    0\ \ \ 0
    \end{aligned}
    \right]),
    (0,\frac{\partial}{\partial x_2},\left[ 
    \begin{aligned}
    0\ \ \ 0\\
    0\ \ \ 0
    \end{aligned}
    \right]),
    (0,0,\left[ 
    \begin{aligned}
    b\ \ \ 0\\
    d\ \ \ 0
    \end{aligned}
    \right]),
    (0,0,\left[ 
    \begin{aligned}
    \lambda\ \ \ 1\\
    0\ \ \ 0
    \end{aligned}
    \right]),
    (0,0,\left[ 
    \begin{aligned}
    0\ \ \ 0\\
    \lambda\ \ \ 1
    \end{aligned}
    \right]),
\end{equation}
and the tangent space of {$\im(j^1(df))$ at} $(x_1,x_2,df_{(x_1,x_2)})$ is generated by
\begin{equation}\label{equ:4.2}
    (\frac{\partial}{\partial x_1},0,\left[ 
    \begin{aligned}
    \frac{\partial^2 f_1}{\partial x_1^2}\ \ \ \frac{\partial^2 f_1}{\partial x_1x_2}\\
    \frac{\partial^2 f_2}{\partial x_1^2}\ \ \ \frac{\partial^2 f_2}{\partial x_1x_2}
    \end{aligned}
    \right]),
    (0,\frac{\partial}{\partial x_2},\left[ 
    \begin{aligned}
    \frac{\partial^2 f_1}{\partial x_1x_2}\ \ \ \frac{\partial^2 f_1}{\partial x_2^2}\\
    \frac{\partial^2 f_2}{\partial x_1x_2}\ \ \ \frac{\partial^2 f_2}{\partial x_2^2}
    \end{aligned}
    \right]).
\end{equation}
Since $S_1$ is a codimensional one submanifold of $J^1(\R^2,\R^2)$, then the vectors in Equation (\ref{equ:4.1}) (\ref{equ:4.2}) generate the tangent space of $J^1(\R^2,\R^2)$ if and only if one of the vectors in Equation (\ref{equ:4.2}) is linear independent of the vectors in Equation (\ref{equ:4.1}). This is equivalent to
\begin{equation*}
    \det\left[
    \begin{aligned}
    \frac{\partial^2 f_1}{\partial x_1^2}\ \ \ \ \ \ b_0\ \ \ \ \ \lambda_0\ \ \ \ \ 0\\
    \frac{\partial^2 f_1}{\partial x_1x_2}\ \ \ \ \ 0\ \ \ \ \ \ 1\ \ \ \ \ \ 0\\
    \frac{\partial^2 f_2}{\partial x_1^2}\ \ \ \ \ \ d_0\ \ \ \ \ 0\ \ \ \ \ \lambda_0\\
    \frac{\partial^2 f_2}{\partial x_1x_2}\ \ \ \ \ 0\ \ \ \ \ \ 0\ \ \ \ \ \ 1\\
    \end{aligned}
    \right]({x_0})\ne0
\end{equation*}
or
\begin{equation*}
    \det\left[
    \begin{aligned}
    \frac{\partial^2 f_1}{\partial x_1x_2}\ \ \ \ \ b_0\ \ \ \ \ \lambda_0\ \ \ \ \ 0\\
    \frac{\partial^2 f_1}{\partial x_2^2}\ \ \ \ \ \ 0\ \ \ \ \ \ 1\ \ \ \ \ \ 0\\
    \frac{\partial^2 f_2}{\partial x_1x_2}\ \ \ \ \ d_0\ \ \ \ \ 0\ \ \ \ \ \lambda_0\\
    \frac{\partial^2 f_2}{\partial x_2^2}\ \ \ \ \ \ 0\ \ \ \ \ \ 0\ \ \ \ \ \ 1\\
    \end{aligned}
    \right]({x_0})\ne0.
\end{equation*}
After plugging in Equation (\ref{equ:4.3}), we have
\[
(\frac{\partial f_2}{\partial x_1}\frac{\partial^2 f_1}{\partial x_1x_2}-\frac{\partial f_2}{\partial x_2}\frac{\partial^2 f_1}{\partial x_1^2}+\frac{\partial f_1}{\partial x_2}\frac{\partial^2 f_2}{\partial x_1^2}-\frac{\partial f_1}{\partial x_1}\frac{\partial^2 f_2}{\partial x_1x_2})({x_0})\ne0
\]
or
\[
(\frac{\partial f_2}{\partial x_1}\frac{\partial^2 f_1}{\partial x_2^2}-\frac{\partial f_2}{\partial x_2}\frac{\partial^2 f_1}{\partial x_1x_2}+\frac{\partial f_1}{\partial x_2}\frac{\partial^2 f_2}{\partial x_1x_2}-\frac{\partial f_1}{\partial x_1}\frac{\partial^2 f_2}{\partial x_2^2})({x_0})\ne0.
\]
These are the same as
\[
\frac{\partial}{\partial x_1}(\det df_{x})({x_0})\ne0
\]
or
\[
\frac{\partial}{\partial x_2}(\det df_{x})({x_0})\ne0.
\]
In other words, $0$ is a regular value of $\det (df)$ at ${x_0}$ {in this case.\\
Similar calculations apply in the situation of Equation (\ref{equ:4.3.1}) and one can get that $\frac{\partial}{\partial x_1}(\det df_{x})({x_0})\ne0$ or $\frac{\partial}{\partial x_2}(\det df_{x})({x_0})\ne0$. Thus the proof is complete.}
\end{proof}

\begin{defn}\label{def:trans}
Suppose that $g=(g_1,g_2):\R^2\rightarrow\R^2\times\R^2$ is a Lagrangian immersion. A point $x\in\R^2$ is called a {\bf transverse bisingular point} if $dg_1$ is transverse to $S_1$ at $x$.
\end{defn}

Suppose that $g=(g_1,g_2):\R^2\rightarrow\R^2\times\R^2$ is a Lagrangian immersion. Recall that the bisingular set of $g$ is defined as $S_1(g_1)=dg_1^{-1}(S_1)$. If $dg_1$ is transverse to $S_1$, then $S_1(g_1)$ is a smooth submanifold of $\R^2$.

\begin{them}\label{thm1}
Let $g=(g_1,g_2):\R^2\rightarrow\R^2\times\R^2$ be a Lagrangian immersion. Suppose that $g(0,0)=(0,0,0,0)$ and $(0,0)$ is a bisingular point. Then there is a number $\delta>0$, an open neighbourhood $U\subset\R^2$ of $(0,0)$, and a regular homotopy of Lagrangian immersions 
\[
(g_1,g_2)(\cdot,t)=g^t(\cdot)=(g_1^t,g_2^t)(\cdot):\R^2\rightarrow\R^2\times\R^2
\]
for $t\in[-\delta,\delta]$, satisfying
\begin{itemize}
    \item the homotopy starts at $g$, i.e., $g^t|_{t=0}=g$,
    \item  0 is a regular value of the maps
    \[
    \det(dg_i^t)(\cdot,\cdot):U\times (-\delta,\delta)\rightarrow\R \ \text{for}\ i=1,2,
    \]
\end{itemize}
\end{them}

\begin{rmk}
    {The term {\bf regular homotopy} in the above theorem means a smooth one parameter family of immersions. This is not a Hamiltonian isotopy since we first identify a tubular neighbourhood of the Lagrangian immersion with its cotangent bundle and then a Hamilton flow is applied in this cotangent bundle. In contrast, a Hamiltonian isotopy should be a perturbation {along} a Hamiltonian vector field in the ambient space $\R^2\times\R^2$.}
\end{rmk}

\begin{proof}
{\bf The first step} is to set up local coordinates. To distinguish the components of $\R^2\times\R^2$, set 
\[
\R^2_+\times\R^2_-:=\R^2\times\R^2.
\]
Since $g$ is a Lagrangian immersion and $(0,0)$ is a bisingular point, then 
\[
\det(dg_1)(0,0)=\det(dg_2)(0,0)=0.
\]
Because 
\[
\rank(dg)_{(0,0)}=2,\ \  \rank(dg_i)_{(x_1,x_2)}\ge1\ \  \forall (x_1,x_2),
\]
then 
\[
\rank(dg_1)_{(0,0)}=\rank(dg_2)_{(0,0)}=1
\]
and 
\[
\ker (dg_1)_{(0,0)}\oplus\ker(dg_2)_{(0,0)}=T_{(0,0)}\R^2, \ \ker (dg_1)_{(0,0)}\cap\ker (dg_2)_{(0,0)}=\{0\}.
\]
Take a local coordinate $(x_1,x_2)$ of $\R^2$ around $(0,0)$ such that 
\[
\ker (dg_2)_{(0,0)}=\Span\langle\frac{\partial}{\partial x_1}\rangle,\ \ \ker (dg_1)_{(0,0)}=\Span\langle\frac{\partial}{\partial x_2}\rangle.
\]
Therefore 
\[
0\ne(dg_1)_{(0,0)}(\frac{\partial}{\partial x_1})\in T_{(0,0)}\R^2_+,\ \ 0\ne(dg_2)_{(0,0)}(\frac{\partial}{\partial x_2})\in T_{(0,0)}\R^2_-.
\]
So there is a local symplectic coordinate $(y_1,y_2,y_3,y_4)$ in $\R^2_+\times\R^2_-$ such that 
\[
(dg_1)_{(0,0)}(\frac{\partial}{\partial x_1})=\frac{\partial}{\partial y_1},\ \ (dg_2)_{(0,0)}(\frac{\partial}{\partial x_2})=\frac{\partial}{\partial y_3}.
\]
Under these local coordinates, we have
\begin{equation*}
    (dg_1)_{(0,0)}= 
    \begin{bmatrix} 
    1&0\\
    0&0 
    \end{bmatrix},
    (dg_2)_{(0,0)}= 
    \begin{bmatrix} 
    0&1\\
    0&0 
    \end{bmatrix}.
\end{equation*}
{Applying the second remark after the Weinstein tubular neighbourhood theorem \ref{thm:tubu}, the immersion $g|_D$ can be extended to a local symplectic diffeomorphism $G$ as in the following commutative diagram, where $D\subset\R^2$ is a small open disc of the origin:
\begin{equation}\label{diag:1}
    \xymatrix{
    T_\varepsilon^*\mathbb R^2|_D \ar[dr]^G & \\
    D \ar[r]^{g|_D\ \ \ \ } \ar[u] & \mathbb R^2_+\times\mathbb R^2_-,}
\end{equation}
here $T_\varepsilon^*\mathbb R^2$ is an $\varepsilon$-tubular neighbourhood of the zero-section of $T^*\R^2$. Then the local coordinate $(x_1,x_2)$ of $D$ can be extended to a local coordinate $(x_1,x_2,x_3,x_4)$ of $T_\varepsilon^*\mathbb R^2|_D$, where $x_3$ and $x_4$ are the coordinates for $dx_1$ and $dx_2$, resp. Then 
\[
G(x_1,x_2,0,0)=g(x_1,x_2)
\]
in these coordinates.\\
{\bf The second step} is to construct the homotopy with the above local coordinates. Recall that given a smooth function $h:D\rightarrow\R$, $dh$ is a Lagrangian submanifold of $T^*\R^2|_D$. If $h$ is compactly supported near the origin, then $h$ can be extended trivially to $\R^2$ and $dh$ coincides with the 0-section of $T_\varepsilon^*\mathbb R^2|_D$ outside the support of $h$. Therefore after extending $h$ to $\R^2$, one can define }
\begin{equation}\label{equ:4.13}
    g^t(x_1,x_2)=G(x_1,x_2,t\frac{\partial h}{\partial x_1}(x_1,x_2),t\frac{\partial h}{\partial x_2}(x_1,x_2)).
\end{equation}
In the coordinates chosen above, the Jacobian of $G$ at $(0,0,0,0)$ can be written as 
\begin{equation*}
    \begin{bmatrix} 
    1&0&a&b\\
    0&0&c&d\\
    0&1&p&q\\
    0&0&r&s
    \end{bmatrix}.
\end{equation*}
Since $dG$ is of full rank, then 
\begin{equation}\label{equ:4.12}
    \det(\begin{bmatrix}
c&d\\
r&s
\end{bmatrix})=cs-rd\ne0,\ and\ s^2+r^2\ne0.
\end{equation}
For fixed $t$, the Jacobian of $g^t=(g_1^t,g_2^t)$ at $(0,0)$ is
\begin{equation*}
    (dg^t)_{(0,0)}=
    \begin{bmatrix} 
    1&0&a&b\\
    0&0&c&d\\
    0&1&p&q\\
    0&0&r&s
    \end{bmatrix}
    \begin{bmatrix} 
    1&0\\
    0&1\\
    th_{x_1x_1}&th_{x_1x_2}\\
    th_{x_1x_2}&th_{x_2x_2}
    \end{bmatrix}.
\end{equation*}
Then the Jacobian for $g_1^t(x_1,x_2)$ at $(0,0)$ is
\begin{equation*}
    (dg_1^t)_{(0,0)}=
    \begin{bmatrix} 
    1&0\\
    0&0
    \end{bmatrix}+t
    \begin{bmatrix} 
    a&b\\
    c&d
    \end{bmatrix}
    \begin{bmatrix} 
    h_{x_1x_1}&h_{x_1x_2}\\
    h_{x_1x_2}&h_{x_2x_2}
    \end{bmatrix},
\end{equation*}
and the Jacobian for $g_2^t(x_1,x_2)$ at $(0,0)$ is
\begin{equation*}
    (dg_2^t)_{(0,0)}=
    \begin{bmatrix} 
    0&1\\
    0&0
    \end{bmatrix}+t
    \begin{bmatrix} 
    p&q\\
    r&s
    \end{bmatrix}
    \begin{bmatrix} 
    h_{x_1x_1}&h_{x_1x_2}\\
    h_{x_1x_2}&h_{x_2x_2}
    \end{bmatrix}.
\end{equation*}
Thus 
\begin{equation}\label{equ:4.4}
    \begin{aligned}
    \det(dg_1^t){(x_1,x_2)}=t(ch_{x_1x_2}+dh_{x_2x_2})+O(t^2),\\ \det(dg_2^t){(x_1,x_2)}=t(rh_{x_1x_1}+sh_{x_1x_2})+O(t^2).
\end{aligned}
\end{equation}

Let 
\[
h=\frac12rx_1^2+sx_1x_2-\frac12rx_2^2,\]
then 
\[h_{x_1x_1}=r,\ \ h_{x_1x_2}=s,\ \ h_{x_2x_2}=-r.
\]
Therefore by Equation (\ref{equ:4.12}),
\[
ch_{x_1x_2}+dh_{x_2x_2}=cs-rd\ne0,\ \ rh_{x_1x_1}+sh_{x_1x_2}=r^2+s^2\ne0.
\]
As a result, 
\begin{equation}\label{equ:cons}
    \frac{\partial}{\partial t}\det (dg_i^t)\ne0,\ i=1,2.
\end{equation}
This implies that 0 is a regular value for both $\det (dg_1^t)$ and $\det (dg_2^t)$. Since $\det(dg_i^t)$ is smooth, there is an interval $(-\delta,\delta)$ and an open neighbourhood {$U\subset D\subset\R^2$} such that  0 is a regular value of the maps
\[
\det(dg_i)(\cdot,\cdot):U\times (-\delta,\delta)\rightarrow\R,
\]
for $i=1,2$.
\end{proof}

In the above proof, first the Weinstein tubular neighbourhood theorem is applied to get an extension {$G:T_\varepsilon^*\mathbb R^2|_D\rightarrow\R^2_+\times\R^2_-$ of $g|_D$ as in Diagram (\ref{diag:1}) for a small open disc $D$ of the origin in $\R^2$. Then we find a smooth function $h:D\rightarrow\R$ compactly supported near the origin to get a homotopy $g^t$ starting at $g$ as in Equation (\ref{equ:4.13}) such that the conclusions of Theorem \ref{thm1} are satisfied.} Let $U$ be a neighbourhood in {$D\subset\R^2$} and $\delta$ be a positive real number as in Theorem \ref{thm1}.

\begin{defn}\label{def:per1}
Call the triple $(U,h,\delta)$ given above {\bf perturbation data of first type}.
\end{defn}

\subsection{Local to global}\label{sec:loctog}
Consider a Lagrangian immersion $(F,g;F_1\times F_2)$, where $F$, $F_1$, $F_2$ are three closed surfaces. Equip $F_1$, $F_2$, $F_1\times F_2$ with symplectic forms $\omega_1$, $\omega_2$, $\omega_1\times(-\omega_2)$ respectively. By using the Weinstein tubular neighbourhood theorem \ref{thm:tubu}, $g$ can be extended to a local symplectic diffeomorphism $G:T_\varepsilon^*F\rightarrow F_1\times F_2$ as in the following diagram:
\begin{equation}\label{graph:Lag}
    \xymatrix{
    T_\varepsilon^*F \ar[dr]^G & \\
    F \ar[r]^{g\ \ \ \ } \ar[u] &F_1\times F_2,}
\end{equation}
where $T_\varepsilon^*F$ is an $\varepsilon$-neighborhood of the zero section of $T^*F$.
In the previous section it was shown that for any point $x\in F$, there is a triple called perturbation data of first type $(U_x,h_x,\delta_x)$ such that the perturbed function has only transverse singularities in $U_x$. Since $F$ is compact, there are finitely many triples $\{(U_j,h_j,\delta_j)\}_{j=1}^N$ such that $F=\cup_jU_j$. {The method in Golubitsky and Guillemin proving transversality (Chapter 2 Lemma 4.6 \cite{golubitsky2012stable})} is applied to show that there are some parameters so that the perturbation functions can be combined to get a global perturbation such that all bisingular points are transverse bisingular points.

\begin{them}\label{thm2}
Assume $F$, $F_1$, $F_2$ are closed surfaces. Equip $F_1$, $F_2$, $F_1\times F_2$ with symplectic forms $\omega_1$, $\omega_2$, $\omega_1\times(-\omega_2)$ respectively. Let $g=(g_1,g_2):F\rightarrow F_1\times F_2$ be a Lagrangian immersion. Then there is a number $\delta>0$ and a regular homotopy of Lagrangian immersions 
\[
(g_1,g_2)(\cdot,t)=g^t=(g_1^t,g_2^t):F\rightarrow F_1\times F_2\ for\ t\in[0,\delta]
\]
such that the followings hold.
\begin{itemize}
    \item The homotopy starts at $g$, i.e. $g^t|_{t=0}=g$.
    \item The homotopy ends with a function $g^\delta$ whose bisingular set is a smooth submanifold of $F$.
\end{itemize}
\end{them}

\begin{rmk}
The number $\delta$ in Theorem \ref{thm2} can be chosen from a dense open subset of a small interval around 0. See the proof for details.
\end{rmk}

\begin{proof}
Assume $x\in F$ is a bisingular point of $g$. Let $G$ be the local symplectomorphism given by the Weinstein tubular neighbourhood theorem as in Diagram (\ref{graph:Lag}). Suppose $(x_1,x_2)$ is a local coordinate around $x$, we can extend this coordinate to a local coordinate $(x_1,x_2,x_3,x_4)$ of $T^*_\varepsilon F$, where $x_3$ and $x_4$ represent the coordinates corresponding to $dx_1$ and $dx_2$, resp. Then 
\[
G(x_1,x_2,0,0)=g(x_1,x_2).
\]
The goal is to find some smooth function 
\[
h:F\rightarrow\R
\]
such that
\begin{equation}\label{equ:G}
    (G_1(x_1,x_2),G_2(x_1,x_2)):=G(x_1,x_2,\frac{\partial h}{\partial x_1}(x_1,x_2),\frac{\partial h}{\partial x_2}(x_1,x_2))
\end{equation} 
has only transverse bisingular points.\\
For any point $x\in F$, there is a triple of perturbation data of first type $(U_x,h_x,\delta_x)$ (Definition \ref{def:per1}) constructed in the proof of Theorem \ref{thm1}. Since $F=\cup_xU_x$ is compact, therefore there is a finite cover $\{U_j\}_{j=1}^N$ of $F$. Then there is a finite family of such triples $\{(U_j,h_j,\delta_j)\}_{j=1}^N$ corresponding to $\{U_j\}_{j=1}^N$. Since there are finite many $j$'s, the $\delta_j$ can be replaced by $\delta:=\min_j(\delta_j)$ in each triple. Let $V_j$ be a subset of $U_j$  such that 
\[
V_j\subset\bar{V}_j\subset U_j\ \  \text{for}\ \ j=1,\ldots,N.
\]
Assume further $F=\cup_jV_j$. Take a set of partition of unity functions $\{\rho_j\}_{j=1}^N$ such that
\begin{itemize}
    \item $\sum_j\rho_j=1$,
    \item $\rho_j|_{V_j}=1$,
    \item $\rho_j=0$ outside $U_j$ for all $j$.
\end{itemize}
Notice that $V_j\subset U_j$, $\rho_jh_j|_{V_j}=h_j|_{V_j}$, and the triple $(U_j,h_j,\delta_j)$ is perturbation data of first type of $g|_{U_j}$. Then $(V_j,\rho_jh_j,\delta)$ is perturbation data of second type for $g|_{V_j}$. Let 
\[
h:=\sum_jt_j\rho_jh_j
\]
and plug $h$ in Equation (\ref{equ:G}) to get $G_1$ and $G_2$, for $(t_1,\ldots,t_N)\in(0,\delta)^N$.\\
According to Lemma \ref{lem:trancri}, only need to show 0 is a regular value for $G_1$ and $G_2$ with some fixed $(t_1,\ldots,t_N)\in(0,\delta)^N$. {To achieve this, one needs to assign a Riemannian structure on these 3 surfaces such that the determinant of smooth maps is well defined so that one can define the following map for $i=1,2$}:
\begin{align*}
    H_i:(0,\delta)^N\times F&\rightarrow \mathbb R\\
    (t_1,\ldots,t_N,x_1,x_2)&\mapsto \det(dG_i)(x_1,x_2).
\end{align*}
First we show that both $H_1$ and $H_2$ have 0 as a regular value. Let $x=(x_1,x_2)\in V_j\subset F$ be a bisingular point. Since $(V_j,\rho_jh_j,\delta)$ is the perturbation data of first type, according to the construction of the perturbation in Theorem \ref{thm1} (see Equation (\ref{equ:cons})),
\[
\frac{\partial}{\partial t_j}\det(dG_i)(x_1,x_2)\ne0,\ \text{for}\ (t_1,\ldots,t_N)=(0,\ldots,0).
\]
This implies that 0 is a regular value for $H_i|_{(0,\delta)^N\times V_j}$ with $\delta$ sufficiently small, $i=1,2$. Because this holds for any bisingular points $x$, therefore 0 is a regular value for both $H_1$ and $H_2$ with small enough $\delta$.\\
It remains to show that there is a vector $t^0:=(t_1^0,\ldots,t^0_N)$ such that 
\begin{equation}\label{equ:van}
    (dH_i)_{(t^0,x_1,x_2)}(v_0)\ne0,
\end{equation}
for all $(x_1,x_2)\in F$ and some $v_0\in T_{(t_0,x_1,x_2)}F\subset T_{(t_0,x_1,x_2)}((0,\delta)^N\times F)$ with $i=1,2$. Since $0$ is a regular value of $H_1$ and $H_2$, then $H_1^{-1}(0)$ and $H_2^{-1}(0)$ are both smooth manifolds. Consider the surjective projections 
\[
\pi_i:H_i^{-1}(0)\rightarrow (0,\delta)^N,\ i=1,2.
\]
By Sard's theorem, the regular value for both maps are dense and open. Take $t^0=(t_1^0,\ldots,t^0_N)$ in the regular value set, then $\pi_i^{-1}(t^0)\subset H_i^{-1}(0)$ is a smooth one dimensional submanifold for $i=1,2$. Suppose $(x_1,x_2)\in \pi_i^{-1}(t^0)$. Given $v_1,\ldots,v_N\in T_{(t^0,x_1,x_2)}H^{-1}_i(0)$, we have
\begin{equation}\label{equ:5}
    (dH_i)_{(t^0,x_1,x_2)}(v_1)=\ldots=(dH_i)_{(t^0,x_1,x_2)}(v_N)=0.
\end{equation}
Since $t^0$ is a regular value of $\pi_i$ for $i=1,2$, assume  that
\begin{equation}\label{equ:6}
    (d\pi_i)_*v_1=\frac\partial{\partial t_1},\ldots,(d\pi_i)_*v_N=\frac\partial{\partial t_N},
\end{equation}
for $i=1,2$. Suppose Equation (\ref{equ:van}) is not true, then
\begin{equation}\label{equ:7}
    (dH_i)_{(t^0,x_1,x_2)}(\frac\partial{\partial x_1})=0\ \text{and}\ (dH_i)_{(t^0,x_1,x_2)}(\frac\partial{\partial x_2})=0,\ i=1,2.
\end{equation}
Equations (\ref{equ:5}),  (\ref{equ:6}), (\ref{equ:7}) imply that
\begin{align*}
    &T_{(t^0,x_1,x_2)}H^{-1}_i(0)\\
    \supset&\Span\langle v_1,\ldots,v_N,\frac\partial{\partial x_1},\frac\partial{\partial x_2}\rangle\\
    =&\Span\langle\frac\partial{\partial t_1},\ldots,\frac\partial{\partial t_N},\frac\partial{\partial x_1},\frac\partial{\partial x_2}\rangle\\
    =&T_{(t^0,x_1,x_2)}((0,\delta)^N\times F),
\end{align*}
which is impossible. Therefore Equation (\ref{equ:van}) must hold. The conclusion follows.
\end{proof}

\section{Immersion of the bisingular set}\label{sec:immerofbisin}
Let $F$, $F_1$, and $F_2$ be closed surfaces. Equip $F_1$, $F_2$, $F_1\times F_2$ with symplectic forms $\omega_1$, $\omega_2$, $\omega_1\times(-\omega_2)$ respectively. The previous section shows that a given Lagrangian immersion $(F,g;F_1\times F_2)$ can be perturbed such that its bisingular set $S$ is a smooth one dimensional submanifold of $F$. The task of this section is to show that $g$ can be perturbed further into $\tilde g=(\tilde g_1,\tilde g_2)$ such that the composition 
\[
S\hookrightarrow F\looparrowright F_1\times F_2\rightarrow F_1
\]
is an immersion except at finitely many points and $j^2(\tilde{g_i})$ is transverse to $S_{1,1}(F,F_i)$ at these points for $i=1$ or 2. Once this is achieved, then Whitney's theorem (Theorem \ref{thm:g}, \ref{thm:2ordtran}) implies that these finitely many points are cusp points for $\tilde g_1$ or $\tilde g_2$ and the remaining points in the bisingular set are bifold points. For the reader's convenience, these theorems are restated below.

\begin{them}[Whitney]
Suppose that $f:X\rightarrow Y$ is a smooth map between surfaces. Assume $df$ is transverse to $S_1$. If the restriction of $f$ to $S_1(f)$ is an immersion, then $x$ is a fold point.
\end{them}

\begin{them}[Whitney, cf. Golubitsky and Guillemin \cite{golubitsky2012stable} Section 4, Chapter 6]
Let $f:X\rightarrow Y$ be a smooth map between surfaces. Suppose $df$ and $j^2(f)$ intersect $S_1$ and $S_{1,1}$ transversely, respectively. Then 
\begin{itemize}
    \item $S_{1,1}(f)$ is a zero dimensional submanifold of $S_1(f)$,
    \item there are local coordinates around $x\in S_{1,1}(f)$ and $f(x)$ such that $f$ locally can be written as
    \[
    f(x_1,x_2)=(x_1,x_1x_2+x_2^3).
    \]
    In particular, $x$ is a {\bf cusp point}.
\end{itemize}
\end{them}

The first step is to prove the local case, where $F=F_1=F_2=\R^2$. The second step is to provide a globalized version, using a similar argument as in Section \ref{sec:loctog}.

\begin{them}\label{thm:locim}
Let $g=(g_1,g_2):\R^2\rightarrow\R^2\times\R^2$ be a Lagrangian immersion. Suppose $(0,0)\in \R^2$ is a transverse bisingular point. Then there is a number $\delta>0$ and regular homotopy of Lagrangian immersions 
\[
g^t=(g_1^t,g_2^t):\R^2\rightarrow\R^2\times\R^2,\ t\in[0,\delta]
\]
such that the followings hold.
\begin{itemize}
    \item The homotopy starts at $g$, i.e. $g^t|_{t=0}=g$.
    \item For each fixed $t$, consider the 2-jet map $j^2(g^t_i):\R^2\rightarrow J^2(\R^2,\R^2)$ of $g^t_i$ for $i=1,2$. When $t$ varies, there is a number $\delta>0$ and an open neighbourhood $U$ of $(0,0)$ such that the following map
    \[
    j^2(g^*_i|_U):(-\delta,\delta)\times U\rightarrow J^2(\R^2,\R^2)
    \]
    is transverse to $S_{1,1}$ for $i=1,2$.
\end{itemize}
\end{them}

\begin{proof}
Let $S$ be the bisingular set. If $g_1|_S$ and $g_2|_S$ are immersions at the origin, then the origin is a bifold singularity, the theorem is true automatically. Suppose this is not the case. Since $g$ is a Lagrangian immersion, then
\[
\rank(dg)=2.
\]
Therefore 
\[
\ker dg_1\cap\ker dg_2=\{0\}.
\]
In particular, either $g_1|_S$ or $g_2|_S$ has to be an immersion at $(0,0)$, and the other one must not be an immersion at $(0,0)$. Assume that $g_1|_S$ is not an immersion at $(0,0)$, and the case where $g_2|_S$ not being an immersion at $(0,0)$ can be proved similarly.\\
To distinguish the first and the second components of $\R^2\times\R^2$, the notation $\R^2_+\times\R^2_-$ is used instead of $\R^2\times\R^2$. {The second remark after Weinstein tubular neighbourhood theorem implies that the immersion $g|_D$ can be extended to a local symplectic diffeomorphism $G$ as in the following commutative diagram, where $D\subset\R^2$ is an open disc of the origin:
\begin{equation}\label{diag:2}
    \xymatrix{
    T_\varepsilon^*\mathbb R^2|_D \ar[dr]^{\ \ G=(G_1,G_2,G_3,G_4)} & \\
    D \ar[r]^{g|_D\ \ \ \ } \ar[u] & \mathbb R^2_+\times\mathbb R^2_-.}
\end{equation}}
{\bf We want to choose coordinates in the domain and the target of $G$ around the origins such that $(dG)_{(0,0,0,0)}$ is as simple as possible.} Because 
\[
\rank (dg)_{(0,0)}=2,\ \text{and}\ \rank(dg_i)_{(0,0)}\ge1, i=1,2.
\]
As a result, 
\[
\rank(dg_1)_{(0,0)}=\rank(dg_2)_{(0,0)}=1
\]
and 
\[
\ker (dg_1)_{(0,0)}\oplus\ker (dg_2)_{(0,0)}=T_{(0,0)}\R^2,\ \ker (dg_1)_{(0,0)}\cap\ker (dg_2)_{(0,0)}=\{0\}.
\]
Take local coordinate $(x_1,x_3)$ of {$D\subset\R^2$} around $(0,0)$ such that 
\[
\ker (dg_2)_{(0,0)}=\Span\langle\frac{\partial}{\partial x_1}\rangle,\ \ker(dg_1)_{(0,0)}=\Span\langle\frac{\partial}{\partial x_3}\rangle, 
\]
thus 
\[
0\ne(dg_1)_{(0,0)}(\frac{\partial}{\partial x_1})\in T_{(0,0)}\R^2_+,\ 0\ne(dg_2)_{(0,0)}(\frac{\partial}{\partial x_3})\in T_{(0,0)}\R^2_-.
\]
Let $\gamma(s)$ be a parametrization of $S$ and $\gamma(0)=(0,0)$. Because $g_1|_S$ is not an immersion at $(0,0)$, a parametrization $\gamma(s)$ of $S$ can be chosen such that $\dot{\gamma}(0)=\frac{\partial}{\partial x_3}$. Thus there are local symplectic coordinates $(y_1,y_2,y_3,y_4)$ of $\R_+^2\times\R^2_-$ around $(0,0)$ such that 
\[
(dg_1)_{(0,0)}(\frac{\partial}{\partial x_1})=\frac{\partial}{\partial y_1},\ (dg_2)_{(0,0)}(\frac{\partial}{\partial x_3})=\frac{\partial}{\partial y_3}.
\]
The local coordinate of {$D\subset\R^2$} can be extended to a local coordinate $(x_1,x_2,x_3,x_4)$ of {$T_\varepsilon^*\mathbb R^2|_D$} such that $x_2$ and $x_4$ represent the coefficients of $dx_1$ and $dx_3$ respectively with the following assumptions:
\[
(dG)_{(0,0,0,0)}(\frac{\partial}{\partial x_2})=\frac{\partial}{\partial y_2},\ (dG)_{(0,0,0,0)}(\frac{\partial}{\partial x_4})=\frac{\partial}{\partial y_4}.
\]
Then
\begin{equation*}
    (dG)_{(0,0,0,0)}=
    \begin{bmatrix} 
    1&0&0&0\\
    0&1&0&0\\
    0&0&1&0\\
    0&0&0&1
    \end{bmatrix}.
\end{equation*}
{\bf Then we construct the perturbation function needed.} Let 
\begin{align*}
    h:\R^2&\rightarrow\R\\
    (x_1,x_3)&\mapsto x_1x_3^2.
\end{align*}
Then 
\[
dh(x_1,x_3)=x_3^2dx_1+2x_1x_3dx_3.
\]
Define 
\begin{equation}\label{equ:5.10}
    g^t(x_1,x_3)=(g_1^t(x_1,x_3),g_2^t(x_1,x_3))=G(x_1,t\frac{\partial h}{\partial x_1},x_3,t\frac{\partial h}{\partial x_3})=(G^t_1,G^t_2,G^t_3,G^t_4).
\end{equation}
Observe that $g^t(0,0)=(0,0,0,0)$. Since 
\[
\frac{\partial^2h}{\partial x_1^2}(0,0)=\frac{\partial^2h}{\partial x_1x_3}(0,0)=\frac{\partial^2h}{\partial x_3^2}(0,0)=0,
\]
so
\begin{equation}\label{equ:5.1}
    (dg_1^t)_{(0,0)}=\left[
    \begin{aligned}
    \frac{\partial G_1}{\partial x_1}\ \ \frac{\partial G_1}{\partial x_2}\ \ \frac{\partial G_1}{\partial x_3}\ \ \frac{\partial G_1}{\partial x_4}\\
    \frac{\partial G_2}{\partial x_1}\ \ \frac{\partial G_2}{\partial x_2}\ \ \frac{\partial G_2}{\partial x_3}\ \ \frac{\partial G_2}{\partial x_4}\\
    \end{aligned}\right]
    \left[
    \begin{aligned}
    1\ \ \ \ \ \ \ \ 0\ \ \ \ \\
    t\frac{\partial^2h}{\partial x_1^2}\ \ t\frac{\partial^2h}{\partial x_1x_3}\\
    0\ \ \ \ \ \ \ \ 1\ \ \ \ \\
    t\frac{\partial^2h}{\partial x_1x_3}\ \ t\frac{\partial^2h}{\partial x_3^2}\ \\
    \end{aligned}\right](0,0)=\left[
    \begin{aligned}
    1\ \ \ 0\\
    0\ \ \ 0\\
    \end{aligned}\right].
\end{equation}
Therefore $(0,0)$ is a bisingular point of $g^t$ for all small $t$.\\
{\bf Finally we prove the theorem using Theorem \ref{thm:2.8}.} By Equation (\ref{equ:5.1}), define
\[
K:=\ker (dg^t_1)_{(0,0)}=\Span\langle\frac{\partial}{\partial x_3}\rangle,L:=\coker(dg^t_1)_{(0,0)}=\Span\langle\frac{\partial}{\partial y_2}\rangle.
\]
Theorem \ref{thm:2.8} implies that 
\[
(0,0)\in S_{1,1}(g^t_1)\iff j^2(g^t_1)(0,0)\in S_{1,1}.
\]
This is equivalent to say that the induced map of $g^t_1$
\[
\tilde g^t_1:K\rightarrow \Hom(K,L)
\]
has one dimensional kernel. The exact calculation as in the proof of Theorem \ref{thm:2.8} shows that
\begin{equation}\label{equ:5.8}
    j^2(g^t_1)(0,0)\in S_{1,1}\iff \tilde g^t_1(\frac{\partial}{\partial x_3})=0\iff \frac{\partial^2 G^t_2}{\partial x_3^2}(0,0)=0.
\end{equation}
By Remark \ref{rmk:2.2}, $S_{1,1}$ is a codimension two submanifold in $J^2(\R^2,\R^2)$ and one of the two codimensions is given by the codimension of $S_1$ in $J^1(\R^2,\R^2)$. Because $j^1(f)$ is assumed to intersect $S_1$ transversely, therefore $j^2(g^*_1)$ is transverse to $S_{1,1}$ at $(0,0)$ if
\begin{equation}\label{equ:5.2}
    \frac{\partial}{\partial t}\frac{\partial^2 G^t_2}{\partial x_3^2}|_{t=0}(0,0)\ne0.
\end{equation}
Note that 
\begin{align*}
    &\frac{\partial^2 G^t_2}{\partial x_3^2}(0,0)\\
    =&\frac{\partial}{\partial x_3}({t}\frac{\partial G_2}{\partial x_2}\frac{\partial^2 h}{\partial x_1\partial x_3}+\frac{\partial G_2}{\partial x_3}+{t}\frac{\partial G_2}{\partial x_4}\frac{\partial^2 h}{\partial x_3^2})(0,0)\\
    =&({t}\frac{\partial G_2}{\partial x_2}\frac{\partial^3 h}{\partial x_1\partial x_3^2}+\frac{\partial^2 G_2}{\partial x_3^2}+2{t}\frac{\partial^2 G_2}{\partial x_3\partial x_4}\frac{\partial^2 h}{\partial x_3^2}+{t}^2\frac{\partial^2 G_2}{\partial x_4^2}(\frac{\partial^2 h}{\partial x_3^2})^2)(0,0)\\
    =&2t+\frac{\partial^2 G_2}{\partial x_3^2}(0,0)
\end{align*}
According to Equation $(\ref{equ:5.8})$,
\[
0=\frac{\partial^2 G^0_2}{\partial x_3^2}(0,0)=\frac{\partial^2 G_2}{\partial x_3^2}(0,0).
\]
Therefore
\[
\frac{\partial^2 G^t_2}{\partial x_3^2}(0,0)=2t.
\]
Combine this with Equation (\ref{equ:5.2}), the theorem is proved.
\end{proof}

In the above proof, first the Weinstein tubular neighbourhood theorem is applied to get an extension {$G:T_\varepsilon^*\mathbb R^2|_D\rightarrow\R^2_+\times\R^2_-$ of $g|_D$ as in Diagram (\ref{diag:2}). Then we find a smooth function $h:D\rightarrow\R$ compactly supported around the origin to get a homotopy $g^t$ starting at $g$ as in Equation (\ref{equ:5.10}) such that the conclusions of Theorem \ref{thm:locim} are satisfied.} Let $U$ be a neighbourhood in {$D\subset\R^2$} and $\delta$ be a positive real number as in Theorem \ref{thm:locim}.

\begin{defn}\label{def:perdat2}
Call the triple $(U,h,\delta)$ given above {\bf perturbation data of second type}.
\end{defn}

The next theorem corresponds to the global case.

\begin{them}\label{main5}
Assume $F$, $F_1$, $F_2$ are closed surfaces. Equip $F_1$, $F_2$, $F_1\times F_2$ with symplectic forms $\omega_1$, $\omega_2$, $\omega_1\times(-\omega_2)$ respectively. Let $g=(g_1,g_2):F\rightarrow F_1\times F_2$ be a Lagrangian immersion between closed surfaces. Suppose $g$ has transverse bisingular points only. Then there is a number $\delta>0$ and an isotopy of Lagrangian immersions 
\[
g^t=(g_1^t,g_2^t):F\rightarrow F_1\times F_2,\ t\in[0,\delta],
\]
such that the followings hold.
\begin{itemize}
    \item The homotopy starts at $g$, i.e. $g^t|_{t=0}=g$.
    \item The elements in the bisingular set are bifold points of $g^\delta$ except finitely many points.
    \item  These finitely many points are cusp points of $g_1^\delta$ or $g_2^\delta$.
\end{itemize}
\end{them}

\begin{proof}
{\bf The first step} is to construct a perturbation function. Let $G$ be the local symplectomorphism after applying the Weinstein tubular neighbourhood theorem to $g$ as in the following commutative diagram.
\begin{equation}
    \xymatrix{
    T_\varepsilon^*F \ar[drr]^{G} & &\\
    F \ar[rr]_{g\ \ \ \ \ \ \ \ } \ar[u]& &F_1\times F_2.}
\end{equation}
Denote $S$ as the bisingular set of $g$. Since $g$ has transverse bisingular points only, then $S$ is a smooth submanifold in $F$. For every $x\in S$, apply Theorem \ref{thm:locim} to get a triple $(U_x,h_x,\delta_x)$ of perturbation data of second type in Definition \ref{def:perdat2}. Since the bisingular set $S$ is compact, then it can be covered by finitely many open set $\{U_j\}_{j=1}^N$ from the collection $\{U_x\}_x$. Let $\{V_j\}_{j=1}^N$ be the set of smaller coordinate charts such that 
\[
V_j\subset\bar{V_j}\subset U_j.
\]
Assume further $S\subset\cup_jV_j$. Set $V_{N+1}:= F-S$, then $F=\cup_{j}^{N+1}V_j$. Take a set of partition of unity functions $\{\rho_j\}_{j=1}^{N+1}$ such that
\begin{itemize}
    \item $\sum_j\rho_j=1$,
    \item $\rho_j|_{V_j}=1$,
    \item $\rho_j=0$ outside $U_j$ for all $j$.
\end{itemize}
Notice that $V_j\subset U_j$, $\rho_jh_j|_{V_j}=h_j|_{V_j}$, and the triple $(U_j,h_j,\delta_j)$ is perturbation data of second type of $g|_{U_j}$. Then $(V_j,\rho_jh_j,\delta_j)$ is perturbation data of first type for $g|_{V_j}$ with $j=1,\ldots, N$. Let 
\[
h_{\Vec{t}}=\sum_{j=1}^Nt_j\rho_jh_j,\ \ {\Vec{t}}:= (t_1,\ldots,t_N)\in(0,\delta_j)^N. 
\]
Because $h_{\Vec{t}}$ is defined and compactly supported in $\cup_j^NV_j$, then $h_{\Vec{t}}$ can be extended as 0 in $F-\cup_j^NV_j$. Therefore $h_{\Vec{t}}$ is a globally defined function
\[
h_{\Vec{t}}:F\rightarrow\R.
\]
Let
\begin{equation}\label{equ:G2}
    g^{\Vec{t}}(x_1,x_3)=(g^{\Vec{t}}_1(x_1,x_3),g^{\Vec{t}}_2(x_1,x_3))=G(x_1,\frac{\partial }{\partial x_1}h_{\Vec{t}},x_3,\frac{\partial }{\partial x_3}h_{\Vec{t}}).
\end{equation} 
Define the following map for $i=1,2$:
\begin{align*}
    H_i:(0,\delta)^N\times F&\rightarrow J^2(F,F_i)\\
    (t_1,\ldots,t_N,x_1,x_3)&\mapsto j^2(g_i^{\Vec{t}})(x_1,x_3).
\end{align*}
{\bf The second step} is to show that both $H_1$ and $H_2$ intersect $S_{1,1}$ transversely at $(0,\ldots,0)\times F$. Let $x=(x_1,x_2)\in V_j\subset F$ be a bisingular point. Since $(V_j,\rho_jh_j,\delta)$ is the perturbation data of second type, according to the construction of the perturbation in Theorem \ref{thm:locim}, then $H_i(0,\ldots,t_j,\ldots,0,x_1,x_3)$ is transverse to $S_{1,1}$ when $t_j$ varies near 0. Since this works for all $(x_1,x_3)\in F$, then both $H_1$ and $H_2$ intersect $S_{1,1}$ transversely at $(0,\ldots,0)\times F$.\\
{\bf The third step} is to show that there is a vector $\Vec{t^0}:=(t_1^0,\ldots,t^0_N)$ such that 
\begin{equation}\label{equ:van1}
    H_i(\Vec{t^0},\cdot,\cdot):F\rightarrow J^2(F,F_i)
\end{equation}
is transverse to $S_{1,1}$. Since $H_1$ and $H_2$ is transverse to $S_{1,1}$, then $H_1^{-1}(S_{1,1})$ and $H_2^{-1}(S_{1,1})$ are both smooth manifolds. Consider the surjective projections 
\[
\pi_i:H_i^{-1}(S_{1,1})\rightarrow (0,\delta)^N,\ i=1,2.
\]
By Sard's theorem, the regular value for both maps are dense and open. Take $\Vec{t^0}=(t_1^0,\ldots,t^0_N)$ in the regular value set, then $\pi_i^{-1}(\Vec{t^0})\subset H_i^{-1}(0)$ is a smooth one dimensional submanifold for $i=1,2$. Suppose $(x_1,x_2)\in \pi_i^{-1}(\Vec{t^0})$. Let $v_1,\ldots,v_N\in T_{(\Vec{t^0},x_1,x_2)}H^{-1}_i(0)$ be vectors such that
\begin{equation}\label{equ:5.5}
    (dH_i)_{(\Vec{t^0},x_1,x_2)}(v_1),\ldots,(dH_i)_{(\Vec{t^0},x_1,x_2)}(v_N)\in T_{H_i(\Vec{t^0},x_1,x_2)}S_{1,1}.
\end{equation}
Since $\Vec{t^0}$ is a regular value of $\pi_i$ for $i=1,2$, assume  that
\begin{equation}\label{equ:5.6}
    (d\pi_i)_*v_1=\frac\partial{\partial t_1},\ldots,(d\pi_i)_*v_N=\frac\partial{\partial t_N},
\end{equation}
for $i=1,2$. Suppose the maps in Equation (\ref{equ:van1}) does not intersect $S_{1,1}$ transversely, then
\begin{equation}\label{equ:5.7}
    (dH_i)_{(\Vec{t^0},x_1,x_2)}(\frac\partial{\partial x_1})\in T_{H_i(\Vec{t^0},x_1,x_2)}S_{1,1}\ \text{and}\ (dH_i)_{(\Vec{t^0},x_1,x_2)}(\frac\partial{\partial x_2})\in T_{H_i(\Vec{t^0},x_1,x_2)}S_{1,1},\ i=1,2.
\end{equation}
Equations (\ref{equ:5.5}),  (\ref{equ:5.6}), (\ref{equ:5.7}) imply that
\begin{align*}
    &T_{(\Vec{t^0},x_1,x_2)}H^{-1}_i(S_{1,1})\\
    \supset&\Span\langle v_1,\ldots,v_N,\frac\partial{\partial x_1},\frac\partial{\partial x_2}\rangle\\
    =&\Span\langle\frac\partial{\partial t_1},\ldots,\frac\partial{\partial t_N},\frac\partial{\partial x_1},\frac\partial{\partial x_2}\rangle\\
    =&T_{(\Vec{t^0},x_1,x_2)}((0,\delta)^N\times F),
\end{align*}
which is impossible. Therefore the maps in Equation (\ref{equ:van1}) must intersect $S_{1,1}$ transversely. Then the remark after Theorem \ref{thm:2ordtran} implies $S_{1,1}(g_i)$ are finitely many cusp points for $i=1,2$, since cusp points are discrete. According to the definition of $S_{1,1}(g_i)$, $S_{1}(g_i)-S_{1,1}(g_i)$ is an immerison under $g_i$ with $i=1,2$. So Theorem \ref{thm:g} implies that elements in $S_{1}(g_i)-S_{1,1}(g_i)$ are fold points, $i=1,2$.
\end{proof}



\begin{rmk}
Let $g=(g_1,g_2):F\rightarrow F_1\times F_2$ be a Lagrangian immersion. There are no local coordinates such that both $g_1$ and $g_2$ are locally written as fold maps, even $g$ satisfying the conclusion of the main Theorem \ref{thm:main}. If this is not true, then at a bisingular point, we have either $g(x_1,x_2)=(x_1,x_2^2,x_2,x_1^2)$ or $g(x_1,x_2)=(x_1,x_2^2,x_2^2,x_1)$. Neither of them is possible since the former contradicts with Lemma \ref{lem:det} and the latter has a non-zero kernel.
\end{rmk}

\section{Transversality of bisingular sets}\label{sec:trabisin}
Assume $F$, $F_1$, $F_2$ are closed surfaces. Equip $F_1$, $F_2$, $F_1\times F_2$ with symplectic forms $\omega_1$, $\omega_2$, $\omega_1\times(-\omega_2)$ respectively. Let $(F,g;F_1\times F_2)$ be a Lagrangian immersion, where $F$, $F_1$, and $F_2$ are closed surfaces. It is shown in the previous sections that $g=(g_1,g_2)$ can be perturbed into a new Lagrangian immersion whose bisingular set is an immersed submanifold under $g_1$ and $g_2$ except at finitely many cusp points. In this section, a further perturbation is performed such that the image of the bisingular set only has transversal intersections after removing the cusp points. The proof is the same local-to-global argument as before.

\begin{them}\label{thm:transloc}
Let $g=(g_1,g_2):\R^2\rightarrow\R^2\times\R^2$ be a Lagrangian immersion. Suppose the bisingular set $S$ is an immersion under $g_1$ and $g_2$. Assume $(0,0)\in S$ and $g(0,0)=(0,0,0,0)$. Then there is a number $\delta>0$ and a regular homotopy of Lagrangian immersions 
\[
g^t=(g_1^t,g_2^t):\R^2\rightarrow\R^2\times\R^2,\ t\in[0,\delta]
\]
such that the followings hold.
\begin{itemize}
    \item The homotopy starts at $g$, i.e. $g^t|_{t=0}=g$.
    \item The bisingular set $S_t$ of $g^t$ is a submanifold of $\R^2$ for all $t$.
    \item The homotopy $g^t$ fixes the origin, i.e. $g^t(0,0)=(0,0,0,0)$ for all $t$.
    \item There is a positive numbers $\delta$ and an open neighbourhood $U$ of $(0,0)$ such that 
    \[
    S_t\cap U=S\cap U
    \]
    and the following map is transverse to the $x$-axis of $\R^2$:
    \[
    g_i^*: (-\delta,\delta)\times\{S\cap U-(0,0)\}\rightarrow \R^2,\ i=1,2.
    \]
\end{itemize}
\end{them}

\begin{proof}
Only need to construct the homotopy $g_1^t$, then $g_2^t$ is constructed similarly.\\
To distinguish the first and the second component of $\R^2\times\R^2$, the notation $\R^2_+\times\R^2_-$ is used instead of $\R^2\times\R^2$. {The second remark after the Weinstein tubular neighbourhood theorem implies that the immersion $g|_D$ can be extended to a local symplectic diffeomorphism $G$ as in the following commutative diagram, where $D\subset\R^2$ is a small open disc. 
\begin{equation}\label{diag:3}
    \xymatrix{
    T_\varepsilon^*\mathbb R^2|_D \ar[dr]^{\ \ G=(G_1,G_2,G_3,G_4)} & \\
    D \ar[r]^{g|_D\ \ \ \ } \ar[u] & \mathbb R^2_+\times\mathbb R^2_-.}
\end{equation}}
Using Darboux's theorem, assume the symplectic form of $\R^2_+\times\R^2_-$ around the origin is $\omega_{std}\times(-\omega_{std})$.\\
{\bf The first step} is to construct a Hamiltonian flow on $\R^2_+\times\R^2_-$ {supported around the origin}. Denote $||\cdot||$ as the standard norm of vectors in $\R^4$. Let 
\[
\rho:\R\rightarrow\R
\]
be a {smooth} bump function such that
\begin{itemize}
    \item $\rho=1$ in $(-1,1)$,
    \item $\rho=0$ outside $(-2,2)$,
\end{itemize} 
Define
\begin{align}\label{equ:6.8}
    h:\R^2_+\times\R^2_-&\rightarrow\ \ \ \ \ \ \ \ \ \R\\
    y=(y_1,y_2,y_3,y_4)&\mapsto\frac{1}{2}\rho(\frac{||y||^2}{\epsilon})(y_1^2+y_2^2),
\end{align}
where $\epsilon>0$ is a small number (determined later). {If $y$ is near the origin, then $\rho=1$, and according to Example \ref{exp:Hamflow}, the Hamiltonian vector field of $h$ is
\[
X=y_2\frac{\partial}{\partial y_1}-y_1\frac{\partial}{\partial y_2},\quad ||y||^2<\epsilon.
\]}
Thus for $||y||^2<\epsilon$, the Hamiltonian flow $\phi_t$ of $h$ is the solution of the following equations:
\begin{equation*}
    \left\{
    \begin{aligned}
    \dot{y_1}=&\ \ y_2\\
    \dot{y_2}=&-y_1
    \end{aligned}
    \right..
\end{equation*}
As a result, the flow is 
\begin{equation}\label{equ:6.1}
    \phi_t(y)=(\phi_t^1(y),\phi_t^2(y))=(y_1\cos(t)+y_2\sin(t),-y_1\sin(t)+y_2\cos(t),y_3,y_4),\quad ||y||^2<\epsilon.
\end{equation}
If $||y||^2>2\epsilon$, then $\rho'=0$. This shows that $X$ is a trivial vector field, therefore
\[
\phi_t(y)=y,\quad ||y||^2>2\epsilon.
\]
{\bf The second step} is to construct a flow on {$T_\varepsilon^*\R^2|_D$ supported in a neighbourhood of $(0,0,0,0)$ using $\phi_t$ and $G$}. Because $G$ is a local symplectomorphism, $\epsilon$ can be chosen small enough such that $G|_{G^{-1}(||y||^2<3\epsilon)}$ is a symplectomorphism. Then 
\[
G^{-1}|_{(||y||^2<3\epsilon)}(\phi_t)
\]
is a flow in $G^{-1}(||y||^2<3\epsilon)$ and trivial near the boundary. This flow can be trivially extended outside of $G^{-1}(||y||^2<3\epsilon)$, so we get a globally defined flow $\varphi_t$ on {$T_\varepsilon^*\R^2|_D$}. Since $\phi_t$ is a symplectomorphism of $\R^2_+\times\R^2_-$ for every $t$, then $\varphi_t$ is a symplectomorphism of {$T_\varepsilon^*\R^2|_D$} for all $t$.\\
{\bf The final step} is to construct the homotopy $g_1^t$ and to prove that $g_1^t$ has the properties listed in the theorem. Define {$g_1^t:\R^2\rightarrow\R^2_+$ as the trivial extension of the following composition:
\begin{equation}\label{equ:6.9}
    D\stackrel{0-section}\longrightarrow T_\varepsilon^*\R^2|_D\stackrel{G\circ\varphi_t}\longrightarrow\R^2_+\times\R^2_-\stackrel{proj}\longrightarrow\R^2_+.
\end{equation}}
Since $\phi_t$ is a rotation near $(0,0,0,0)$ according to Equation $($\ref{equ:6.1}$)$, then
\[
\phi_t(\{y|||y||^2<\epsilon\})\subset\{y|||y||^2<\epsilon\}.
\]
Let
\[
U=G^{-1}(\phi_t(||y||^2<\epsilon))=G^{-1}(||y||^2<\epsilon),
\]
then
\[
g_1^t(x)=\phi_t\circ g_1(x),\quad x\in U.
\]
Clearly $g^t_1$ is a {homotopy} starting at $g_1$. The second property of the theorem is true if $t$ is small enough. Equation (\ref{equ:6.1}) shows that $g_1^t|_U$ is the composition of $g_1$ with a rotation centered at the origin.
Because rotations are diffeomorphisms, the singular set $S_t\cap U$ of $g_1^t|_U$ is $S\cap U$. This proves the third and the first half of the forth properties.\\
For the second half of the forth property, if there are no points $(0,0)\ne x\in U$ such that $g_1(x)$ is on the $x$-axis, then the forth property holds automatically. Otherwise, let
\[
g_1(x)=(y_1,0),\quad (0,0)\ne x\in U.
\]
Because $g_1|_S$ is an immersion at $(0,0)$, $U$ can be shrunk if necessary such that $y_1\ne0$. Equation (\ref{equ:6.1}) implies
\[
\frac{d}{dt}|_{t=0}g^t_1(x)=\frac{d}{dt}|_{t=0}(y_1\cos(t),-y_1\sin(t))=(0,-y_1).
\]
Therefore there is a positive numbers $\delta$ such that the following map
\[
    g_1^*: (-\delta,\delta)\times\{S\cap U-(0,0)\}\rightarrow \R^2_+.
\]
is transverse to the $x$-axis at all $(0,x)\in(-\delta,\delta)\times\{S\cap U-(0,0)\}$. The proof is complete.
\end{proof}

In the above proof, {first the Weinstein tubular neighbourhood theorem is applied to get an extension $G:T_\varepsilon^*\mathbb R^2|_D\rightarrow\R^2_+\times\R^2_-$ of $g$ as in Diagram (\ref{diag:3}). Then we find a Hamiltonian flow $\varphi_t$ on $T_\varepsilon^*\mathbb R^2$ to get a homotopy starting at $g=(g_1,g_2)$ as in Equation (\ref{equ:6.9})} (This equation only defines a homotopy of $g_1$, but the same argument produces a homopty for $g_2$. Combine these we can construct the homotopy of $g$) such that the conclusions of Theorem \ref{thm:transloc} are satisfied. Let $U$ be a neighbourhood in {$D\subset\R^2$} and $\delta$ be a positive real number as in Theorem \ref{thm:transloc}.

\begin{defn}\label{def:perdat3}
Call the triple $(U,\varphi_t,\delta)$ given above {\bf perturbation data of third type}.
\end{defn}

\begin{rmk}\label{rmk:6.1}
If the perturbation function $h$ in Equation (\ref{equ:6.8}) is replaced by
\begin{align*}
     h:\R^2_+\times\R^2_-&\rightarrow\ \ \ \ \ \ \ \ \ \R\\
    y=(y_1,y_2,y_3,y_4)&\mapsto\frac{1}{2}\rho(\frac{||y||^2}{\epsilon})(y_1+y_2),
\end{align*}
then the corresponding Hamiltonian flow near the origin is
\[
\phi_t(y)=(\phi_t^1(y),\phi_t^2(y))=(y_1+t,y_2-t,y_3,y_4).
\]
The same argument in the proof of Theorem \ref{thm:transloc} shows that there is a number $\delta>0$ and a regular homotopy of Lagrangian immersions starting at $g$ such that $g^\delta$ misses the origin.
\end{rmk}

\begin{them}\label{main6}
Let $(F_1,\omega_1)$ and $(F_2,\omega_2)$ be two closed symplectic surfaces. Given a smooth Lagrangian immersion 
\[
g=(g_1,g_2):F\rightarrow F_1\times F_2,
\]
assume the bisingular set $S$ consists of bifold points and finitely many cusp points. Then there is a positive number $\delta$ and a regular homotopy class of Lagrangian immersions 
\[
g^t=(g_1^t,g_2^t):F\rightarrow F_1\times F_2\ \text{for}\ t\in[0,\delta]
\]
such that the followings hold.
\begin{itemize}
    \item The homotopy starts at $g$, i.e. $g^t|_{t=0}=g$.
    \item Each function $g_1^t|_S$ and $g_2^t|_S$ in the homotopy has only finitely many transversal self intersections for all $t\in(0,\delta]$.
\end{itemize}
\end{them}

\begin{proof}
Only need to construct the homotopy $g_1^t$, then $g_2^t$ is constructed similarly. Recall that the bisingular set $S$ contains only finitely many cusp points, then we can perform a perturbation such that these cusp points are not in the self-intersections under $g_1|_S$ and $g_2|_S$ (as in Remark \ref{rmk:6.1}). Therefore it can be assumed further the bisingular set has no cusp points.\\
{\bf The first step} is to fix the perturbation data. Let $G$ be the local symplectomorphism after applying the Weinstein tubular neighbourhood theorem to $g$ as in the following commutative diagram.
\begin{equation*}
    \xymatrix{
    T_\varepsilon^*F \ar[drr]^{G=(G_1,G_2)} & &\\
    F \ar[rr]_{g\ \ \ \ \ \ \ \ } \ar[u]& &F_1\times F_2.}
\end{equation*} Take an open set $V\subset F_1$ such that $\bar{V}$ is compact. Suppose $g_1^{-1}(V)$ has two disjoint connected components $V_1$ and $V_2$ such that 
\[
g_1(V_1\cap S)\cap g_1(V_2\cap S)\ne\emptyset.
\]
The case where $g_1^{-1}(V)$ has more than two connected components is dealt with similarly. If $x\in V_1\cap S$ and $x'\in V_2\cap S$ are two points such that 
\[
g_1(x)=g_1(x')\in V.
\]
Without loss of generality, one can assume $V_1$, $V_2$, $V$ are local charts centered at $x$, $x'$, $g_1(x)$ respectively. Since $g_1$ is an immersion at $x'$, the local chart $V$ can be assumed further such that $g_1(S\cap V_2')$ is the $x$-axis of $V'$. Under these local charts, Theorem \ref{thm:transloc} can be applied to get a triple $(U_x,\varphi_t^x,\delta_x)$ of perturbation data of third type (the symplectomorphism $G$ is needed) with 
\[
U_x\subset \Bar{U}_x\subset V_1.
\]
If $x\in V_1\cap S$ such that $g_1(x)$ has a unique preimage, then take the triple $(U_x,\varphi_t^x,\delta_x)$ corresponding to the trivial perturbation.\\
{\bf The second step} is to construct a perturbation such that the image of the singular set in $V$ has only transversal self-intersections. Since 
\[
\bar{V}_1\subset\cup_x(U_x-x),
\]
a finite subcover $\{U_j\}_{j=1}^N$ can be found such that
\[
\bar{V}_1\subset\cup_j(U_j-x_j),
\] 
where $x_j$ is the center of the local chart $U_j$. Denote the triple of perturbation data corresponding to $U_j$ as $(U_j,\varphi_{t_j},\delta_j)$, for $j=1,\ldots,N$.
Define
\begin{equation*}
    \begin{aligned}
    g_1^{\Vec{t}}:F&\rightarrow \ \ \ \ \ \ \ \ \ \ \ F_1\\
    x&\mapsto G_1\circ\varphi_{t_N}\circ\ldots\circ\varphi_{t_1}(x),
    \end{aligned}
\end{equation*} 
where $\Vec{t}=(t_1,\ldots,t_N)$. Denote by $S_{\Vec{t}}$ the singular set of $g^{\Vec{t}}_1$. Let
\[
\gamma^{\Vec t}(s):[0,1]\rightarrow F
\]
be the parametrization of $S_{\Vec t}$.\\
The following shows that the image of $S_{\Vec t}$ in $V$ has only transversal self-intersections. Define
\begin{align*}
    H:\Pi_j^N(-\delta_j,\delta_j)\times[0,1]&\rightarrow\ \ \ \ \  \bar V\\
    (\Vec t,s)\ \ \ \ \ \ \ \ \ &\mapsto g^{\Vec t}_1(\gamma^{\Vec t}(s)).
\end{align*}
Since $\{U_j-x_j\}_{j=1}^N$ is an open cover, there is a number $j$ such that $\gamma^{\Vec 0}(s)\in (U_j-x_j)$. Therefore $H$ is transverse to $g_1(S\cap V_2)$ at $(0,\ldots,0,s)$ ($t_j$ direction gives the transversality). Since this holds for all $s$, then there is a number $\delta>0$ such that 
\begin{equation}\label{equ:van2}
    H:(-\delta,\delta)^N\times[0,1]\rightarrow\bar V
\end{equation}
intersects $g_1(S\cap V_2)$ transversely. Consider the projection
\[
\pi:H^{-1}(g_1(S\cap V_2))\rightarrow (-\delta,\delta)^N.
\]
By Sard's theorem, the regular value for $\pi$ is dense and open. Take $\Vec t^0=(t_1^0,\ldots,t^0_N)$ in the regular value set, then $\pi^{-1}(\Vec t^0)\subset H^{-1}(0)$ is a smooth one dimensional submanifold. The claim is that this one dimesional submanifold is transverse to $g_1(S\cap V_2)$. Suppose $s\in \pi^{-1}(\Vec{t^0})$. Let $v_1,\ldots,v_N\in T_{(\Vec{t^0},x_1,x_2)}H^{-1}_i(0)$ be vectors such that
\begin{equation}\label{equ:6.5}
    (dH)_{(\Vec{t^0},s)}(v_1),\ldots,(dH)_{(\Vec{t^0},s)}(v_N)\in \Span\langle\frac{\partial}{\partial x}\rangle.
\end{equation}
Since $\Vec{t^0}$ is a regular value of $\pi$ for $i=1,2$, assume  that
\begin{equation}\label{equ:6.6}
    (d\pi)_*v_1=\frac\partial{\partial t_1},\ldots,(d\pi)_*v_N=\frac\partial{\partial t_N},
\end{equation}
for $i=1,2$. Suppose the maps in Equation (\ref{equ:van2}) does not intersect the $x$-axis transversely, then
\begin{equation}\label{equ:6.7}
    (dH)_{(\Vec{t^0},s)}(\frac\partial{\partial s})\in \Span\langle\frac{\partial}{\partial x}\rangle.
\end{equation}
Equations (\ref{equ:6.5}),  (\ref{equ:6.6}), (\ref{equ:6.7}) imply that
\begin{align*}
    &T_{(\Vec{t^0},s)}H^{-1}(\Span\langle\frac{\partial}{\partial x}\rangle)\\
    \supset&\Span\langle v_1,\ldots,v_N,\frac\partial{\partial s}\rangle\\
    =&\Span\langle\frac\partial{\partial t_1},\ldots,\frac\partial{\partial t_N},\frac\partial{\partial s}\rangle\\
    =&T_{(\Vec{t^0},s)}((0,\delta)^N\times [0,1]),
\end{align*}
which is impossible.\\
{\bf The final step} is to use the induction argument to finish the proof. The image of the singular set can be covered by open sets $\{W_j\}_{j=1}^M$, where $g_1^{-1}(W_j)$ has disjoint connected components. Then apply the second step to $W_1$ such that the image of the singular set of the perturbed function only has self-intersections in $W_1$. Next apply the second step to $W_2$ such that the image of the singular set of the perturbed function has self-intersections in $W_2\cup W_1$ (this can be achieved if the perturbation is small enough). As the process goes on, we get a perturbation such that the image of the singular set of the perturbed function only has transversal self-intersections in $\cup_j W_j$.
\end{proof}

\section{The proof of the main theorem}
The goal of this section is to prove the main theorem (Theorem \ref{thm:main}). In fact, we are going to apply the conclusions of the previous sections to get a stronger version of Theorem \ref{thm:main}.

\begin{them}
Assume $F$, $F_1$, $F_2$ are closed surfaces. Equip $F_1$, $F_2$, $F_1\times F_2$ with symplectic forms $\omega_1$, $\omega_2$, $\omega_1\times(-\omega_2)$ respectively. Let $g=(g_1,g_2):F\rightarrow F_1\times F_2$ be a Lagrangian immersion. Then there is a number $\delta>0$ and a regular homotopy of Lagrangian immersions starting at $g$:
\[
g^t=(g_1^t,g_2^t):F\rightarrow F_1\times F_2\ for\ t\in[0,\delta],\ \ g^t|_{t=0}=g,
\]
such that the followings hold at $g^t|_{t=\delta}$:
\begin{itemize}
    \item $g^\delta$ is a function whose bisingular set $S$ is a smooth submanifold of $F$,
    \item the elements in $S$ are bifold points of $g^\delta$ except finitely many cusp points,
    \item each function $g_1^\delta|_S$ and $g_2^\delta|_S$ in the homotopy has only finitely many transversal self intersections.
\end{itemize}
\end{them}

\begin{proof}
We first apply Theorem \ref{thm2} to get a real number $\delta_1>0$ such that $g^{\delta_1}$ is a function whose bisingular set is a smooth submanifold of $F$. Then Theorem \ref{main5} can be applied to get a a real number $\delta_2>0$ such that not only the bisingular set $S$ of $g^{\delta_1+\delta_2}$ is a smooth submanifold, but also the elements in $S$ are bifold points of $g^{\delta_1+\delta_2}$ except finitely many cusp points. Finally, we use Theorem \ref{main6} to find a number $\delta_3>0$ such that all the properties in the conclusion of the theorem hold with $\delta=\delta_1+\delta_2+\delta_3$.
\end{proof}

\section{{Lagrangian correspondences and Lagrangian compositions on closed surfaces}}\label{sec:Lagcomp}
This section gives the definitions of Lagrangian correspondences and Lagrangian compositions. The concept of Lagrangian correspondences is needed to define Lagrangian compositions. Lagrangian correspondences are special cases of Lagrangian immersions, as in the definition below. When both terms--Lagrangian correspondence and Lagrangian immersion--apply, we use Lagrangian correspondence to stress compositions of Lagrangian immersions.\\
{We begin this section by discussing general Lagrangian correspondences in symplectic manifolds and subsequently narrow our focus to the situation where the symplectic manifolds are closed surfaces. Because Lagrangian immersions of curves and Lagrangian correspondences between closed surfaces are the specific cases required in Section \ref{sec:Lagcov}.}

\begin{defn}
Given symplectic manifolds $(X_0, \omega_0)$ and $(X_1, \omega_1)$, a {\bf Lagrangian correspondence} from $X_0$ to $X_1$ is a Lagrangian immersion
\[
l=(l_0,l_1):L\rightarrow X_0
\times X_1,
\]
where $X_0\times X_1$ is equipped with the symplectic structure $\omega_0\times(-\omega_1)$.
\end{defn}

\begin{defn}\label{def:lacom}
Let $(X_0, \omega_0)$, $(X_1, \omega_1)$, $(X_2, \omega_2)$ be three symplectic manifolds. Then $(X_0\times X_1,\omega_0\times (-\omega_1))$, $(X_1\times X_2,\omega_1\times (-\omega_2))$, $(X_0\times X_2,\omega_0\times (-\omega_2))$ are symplectic manifolds. Assume that 
\[
l^0=(l_0^0,l^0_1):L_{01}\rightarrow X_0
\times X_1\quad and\quad l^1=(l_1^1,l^1_2):L_{12}\rightarrow X_1
\times X_2
\]
are Lagrangian correspondences. The {\bf Lagrangian composition} $L_{01}\circ L_{12}$ of $L_{01}$ and $L_{12}$ is defined as
\begin{equation}\label{equ:lagc}
    l^0\circ l^1:=(l^0_0,l^1_2):L_{01}\circ L_{12}\rightarrow X_0\times X_2,
\end{equation}
where 
\[
L_{01}\circ L_{12}:=L_{01}\times_{X_1} L_{12}=\{(x,y)\in L_{01}\times L_{12}|l^0_1(x)=l^1_1(y)\}.
\]
\end{defn}

\begin{rmk}\label{rmk:lagcom}
As a matter of fact, 
\[
L_{01}\circ L_{12}=(l^0\times l^1)^{-1}(X_0\times\Delta_{X_1}\times X_2),
\]
where 
\[
\Delta_{X_1}:=\{(x,x)|x\in X_1\}
\]
is the diagonal of $X_1\times X_1$. Therefore, if $l^0\times l^1$ is transverse to $X_0\times\Delta_{X_1}\times X_2$, then $L_{01}\circ L_{12}$ is a smooth manifold according to the transversality theorem.
\end{rmk}

The composition $(L_{01}\circ L_{12},l^0\circ l^1;X_0\times X_1)$ in Definition \ref{def:lacom} is not always a Lagrangian immersion. The following proposition gives a sufficient criterion.

\begin{prop}[\cite{wehrheim2010functoriality}]\label{prop:lagco}
Let $(X_0, \omega_0)$, $(X_1, \omega_1)$ and $(X_2, \omega_2)$ be three symplectic manifolds. Then $(X_0\times X_1,\omega_0\times (-\omega_1))$, $(X_1\times X_2,\omega_1\times (-\omega_2))$ and $(X_0\times X_2,\omega_0\times (-\omega_2))$ are symplectic manifolds. Suppose that 
\[
\Delta_{X_1}=\{(x,x)|x\in X_1\}\subset X_1\times X_1
\]
is the diagonal and 
\[
l^0=(l_0^0,l^0_1):L_{01}\rightarrow X_0
\times X_1,l^1=(l_1^1,l^1_2):L_{12}\rightarrow X_1
\times X_2
\]
are Lagrangian correspondences. If $l^0\times l^1$ is transverse to $X_0\times\Delta_{X_1}\times X_2$, then 
\begin{itemize}
    \item $L_{01}\circ L_{12}$ is a smooth manifold,
    \item $(L_{01}\circ L_{12},l^0\circ l^1;X_0\times X_2)$ is a Lagrangian immersion.
\end{itemize}
\end{prop}

\begin{defn}\label{def:composable}
Given Lagrangian correspondences $L_{01}$ and $L_{12}$ in $(X_0
\times X_1,\omega_0\times (-\omega_1))$ and $(X_1
\times X_2,\omega_1\times (-\omega_2))$ respectively, these two Lagrangian correspondences $L_{01}$ and $L_{12}$ are said to be {\bf composable} if all the assumptions in Proposition \ref{prop:lagco} are satisfied.
\end{defn}

{Let $X_1=F_1$, $X_2=F_2$, and $L_{12}=F$ be three closed {oriented} surfaces. By Example \ref{exp:surface}, $F_1$ and $F_2$ are symplectic manifolds. Because every immersed curve on closed surfaces is a Lagrangian immersion. If $L_1\looparrowright F_1$ is an immersed curve, then $L_1$ can be regarded as a Lagrangian correspondence from $\{pt\}$ to $F_1$. Let $F\looparrowright F_1\times F_2$ be a Lagrangian correspondence from a closed surface. If $L_1$ and $F$ are composable, then $L_1\circ F\looparrowright F_2$ is an immersed curve. Similarly, if $L_2$ is an immersed curve into $F_2$, then $L_2$ is a Lagrangian correspondence from $F_2$ to $\{pt\}$. If $L_2$ is composable with $F$, then $F\circ L_2\looparrowright F_1$ is an {immersed} curve. {\bf These are the two cases we focus on in the rest part of the paper.}}


\begin{prop}\label{prop:sg}
{Assume $(X_1,\omega_1)$ and $(X_2,\omega_2)$ are closed symplectic surfaces.} Let $L_1$, $L_{12}$, $L_2$ be Lagrangian immersions in $(X_1,\omega_1)$, $(X_1\times X_2,\omega_1\times(-\omega_2))$, $(X_2,\omega_2)$, resp. Suppose $L_1$ and $L_{12}$, $L_{12}$ and $L_2$ are composable. {If moreover $L_1$ and $L_{12}\circ L_2$, $L_1\circ L_{12}$ and $L_2$ both intersect transversely}, then there is a canonical bijection between {the elements of $L_1\times_{X_1}(L_{12}\circ L_2)$ and $(L_1\circ L_{12})\times_{X_2}L_2$ $($Recall that the fiber product represents the intersection of immersions, see Definition \ref{def:lagint} and Remark \ref{rmk:lagint}$)$.}
\end{prop}

\begin{proof}
Let 
\[
l^1:L_1\rightarrow X_1, l=(l_1,l_2):L_{12}\rightarrow X_1\times X_2,l^2:L_{2}\rightarrow X_2
\]
be the Lagrangian immersions. According to the definition of fiber products,
\begin{align*}
    L_1\times_{X_1}(L_{12}\circ L_2)
    &=L_1\times_{X_1}\{(y,z)\in L_{12}\times L_2|l_2(y)=l^2(z)\}\\
    &=\{(x,y,z)\in L_1\times L_{12}\times L_2|l^1(x)=l_1(y),l_2(y)=l^2(z)\}.
\end{align*}
Similarly,
\begin{align*}
(L_1\circ L_{12})\times_{X_2} L_2
    &=\{(x,y)\in L_{1}\times L_{12}|l^1(x)=l_1(y)\}\times_{X_2} L_2\\
    &=\{(x,y,z)\in L_1\times L_{12}\times L_2|l^1(x)=l_1(y),l_2(y)=l^2(z)\}.   
\end{align*}
The proposition follows.
\end{proof}

{The key point of finding Lagrangian compositions is to decide its domain first. The idea is to regard the domain of the Lagrangian composition as a part of the Lagrangian correspondence $L_{12}$ by using the fiber product $L_1\times_{X_1}L_{12}$. For instance, in Example \ref{exp:8.10}, $X_1=F_1$, $X_2=F_2$, $L_{12}=F$, and $L_1$ is the blue curve in the lower left picture of Figure 3. Then the domain of the Lagrangian composition $L_1\circ L_{12}$ is the blue curve in the upper left picture. The specific map of $L_1\circ L_{12}$ is induced by the second component of $L_{12}\looparrowright X_1\times X_2$. More examples can be found in the next section.}

\begin{rmk}
    {The concept Lagrangian correspondence was first proposed by Weinstein in \cite{weinstein1977lectures} and \cite{weinstein2006symplectic} aimed at using Lagrangian submanifolds to construct morphisms of the category whose objects are symplectic manifolds. But in general the composition of two Lagrangian submanifolds is not a manifold except if some transversality conditions are satisfied (e.g. in the case where Lagrangian immersions are replaced by Lagrangian submanifolds and the conditions in Proposition \ref{prop:lagco} are satisfied).}
\end{rmk}

\section{Lagrangian immersions given by covering maps}\label{sec:Lagcov}

{Let $F_1,F_2$ be closed symplectic surfaces and $L_i\looparrowright F_i$, $F\looparrowright F_1\times F_2$ be Lagrangian immersions, $i=1,2$. The aim of this section is to use the Lagrangian correspondence $F$ between the two closed surfaces $F_1,F_2$ to compare the oriented bigons connecting the points in $L_1\times_{F_1}(F\circ L_2)$ $($Definition \ref{def:lagint} and Remark \ref{rmk:lagint}$)$ with its boundary in $L_1$ and $F\circ L_2$ with the oriented bigons connecting the points in $(L_1\circ F)\times_{F_2} L_2$ with its boundary in $L_1\circ F$ and $L_2$.} Theorem \ref{thm:main} produces a way to calculate the Lagrangian compositions when such Lagrangian correspondences are generic (i.e. satisfy the conclusion of Theorem \ref{thm:main}). This also provides a hint of how to recover the information of {the oriented bigons connecting points in $(L_1\times L_2)\times_{F_1\times F_2}F$ with its bounday in $L_1\times L_2$ and $F$} (Subsection \ref{subsec:8.2}). To achieve the aim, it is necessary first to understand the case where the Lagrangian correspondence has no bisingularities (Subsection \ref{subsec:8.1}).

\begin{defn}
    {Let $(X,\omega)$ be a symplectic manifold and $L\looparrowright X$, $L'\looparrowright X$ be Lagrangian immersions. {\bf A bigon of $\mathbf{L\times_X L'}$} is an oriented immersed unit disc satisfying the following conditions:
    \begin{itemize}
        \item its boundary is the unit circle, the map of the upper half circle can be {lifted} to $L$ and the map of the lower half circle can be {lifted} to $L'$,
        \item the two points $(\pm1,0)$ on the boundary circle {are} mapped to two different points $x,y\in L\times_X L'$,
        \item the image of the disc near $x,y$ are convex corners $($Figure 3$)$.
    \end{itemize}}
\end{defn}

\begin{figure}[H] 
\centering 
\includegraphics[width=0.2\textwidth]{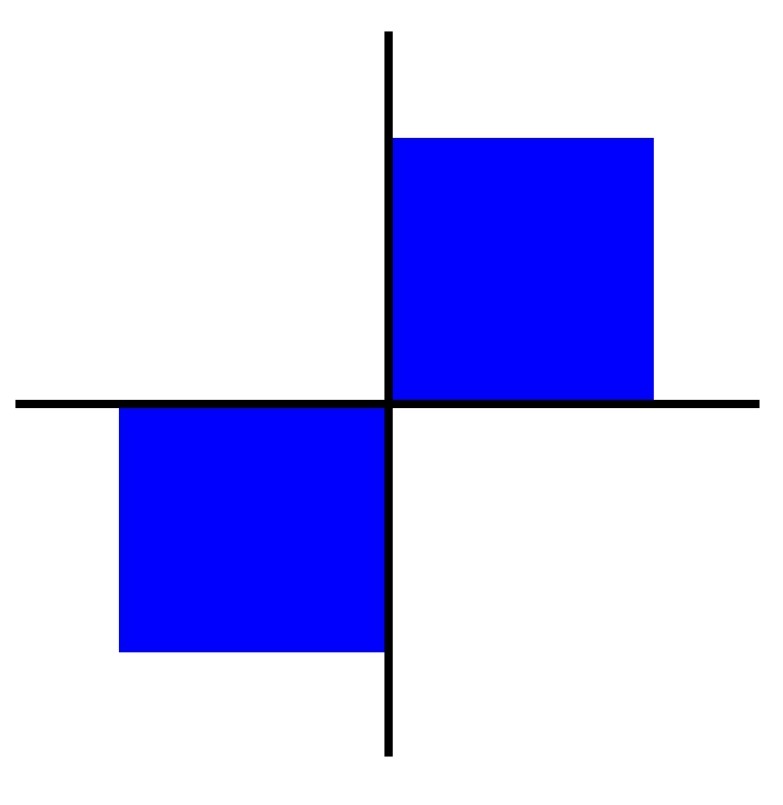}
\caption{{The convex corners are marked in blue. The blue in the first quadrant corresponds to the part of the bigon near $(-1,0)$ and the blue in the third quadrant corresponds to the part of the bigon near $(1,0)$. The horizontal line is a part of $L$ and the vertical line is a part of $L'$. }}
\end{figure}

\begin{rmk}
    {In this section, the above definition is constrained to the cases where:
    \begin{itemize}
        \item $X=F_1,L=L_1,L'=F\circ L_2$,
        \item $X=F_2,L=L_1\circ F,L'=L_2$.
        \item $X=F_1\times F_2,L=L_1\times L_2,L'=F$.
    \end{itemize}}
\end{rmk}

{With additional assumptions, the points in $(L_1\circ F)\times_{F_2} L_2$ form the generators of the Floer chain complex $\CF(L_1\circ F, L_2)$ whose boundary map is obtained by counting bigons of $(L_1\circ F)\times_{F_2}L_2$, and similarly for $\CF(L_2,F\circ L_2)$. For example, see Abouzaid \cite{abouzaid2008fukaya} in the case where the genus of the surfaces are greater than 1 and the Lagrangian immersions are unobstructed.}

{More generally, Wehrheim and Woodward \cite{wehrheim2010quilted} introduced the quilted Lagrangian Floer complex $\CF^Q(L_1,F,L_2)$ whose generators are $(L_1\times L_2)\times_{F_1\times F_2}F$. Their definition requires strict restrictions on $F,L_i,F_i,i=1,2$, and they  proved the isomorphism between the Lagrangian Floer homology group and the corresponding quilted Lagrangian Floer homology group.}

{In this paper, we only compare the bigons of $L_1\times_{F_1}(F\circ L_2)$ with the bigons of $(L_1\circ F)\times_{F_2}L_2$.}

\subsection{Lagrangian immersions without bisingularities}\label{subsec:8.1}
Let $g_1: F\rightarrow F_1$ and $g_2:F\rightarrow F_2$ be two covering maps between closed surfaces. Denote $\omega_1$, $\omega_2$, $\omega_1\times(-\omega_2)$ as the symplectic forms on $F_1$, $F_2$, $F_1\times F_2$ respectively. Assume that
\[
(g_1,g_2):F\rightarrow F_1\times F_2
\] 
is a Lagrangian correspondence from $F_1$ to $F_2$. Suppose that $L_1\looparrowright F_1$ and $L_2\looparrowright F_2$ are two immersed curves, then according to Example \ref{exp:2.1}, $L_1$ and $L_2$ are Lagrangian immersions into $F_1$ and $F_2$, resp. Theorem \ref{prop:sg} shows that {there is a canonical identification between $(L_1\circ F)\times_{F_2}L_2$ and $L_1\times_{F_1}(F\circ L_2)$}. This subsection is about the relation of {the bigons of $L_1\times_{F_1}(F\circ L_2)$ with the bigons of $(L_1\circ F)\times_{F_2} L_2$}.

\begin{thm}\label{1}
Let $F$, $F_1$, $F_2$ be three closed surfaces. Denote $\omega_1$, $\omega_2$, $\omega_1\times(-\omega_2)$ as the symplectic forms on $F_1$, $F_2$, $F_1\times F_2$ respectively. Suppose that
\[
g=(g_1,g_2): F\rightarrow F_1\times F_2
\]
is a Lagrangian correspondence from $F_1$ to $F_2$ and both $g_1$ and $g_2$ are covering maps. {Then} two immersed curves 
\[
L_1\looparrowright F_1\ \text{and}\ L_2\looparrowright F_2
\]
are composable with  $F\looparrowright F_1\times F_2$ and there is a {canonical bijection between the points in $(L_1\circ F)\times_{F_2} L_2$ and the points $L_1\times_{F_1} (F\circ L_2)$. Moreover this bijection can be extended to a bijection between the bigons $(L_1\circ F)\times_{F_2} L_2$ and the bigons of $L_1\times_{F_1} (F\circ L_2)$.}
\end{thm}

\begin{rmk}
    {The elements of $(L_1\circ F)\times_{F_2} L_2$ and  $L_1\times_{F_1} (F\circ L_2)$ generate the Lagrangian Floer chain groups $\CF(L_1\circ F,L_2)$ and $\CF(L_1,F\circ L_2)$ respectively. Moreover the boundary maps of these two chain groups are defined by counting the number of bigons of $(L_1\circ F)\times_{F_2} L_2$ and  $L_1\times_{F_1} (F\circ L_2)$. If all the surfaces have their genus greater than 1 and all the immersed curves are unobstructed, then according to Abouzaid \cite{abouzaid2008fukaya}, $\CF(L_1\circ F,L_2)$ and $\CF(L_1,F\circ L_2)$ are chain complexes, and the above theorem shows that the homology groups of $\CF(L_1\circ F,L_2)$ and $\CF(L_1,F\circ L_2)$ are isomorphic.}
\end{rmk}

This theorem is obtained by applying the following lemmas.

\begin{lem}
Let $g:F\rightarrow F'$ be a $p$-fold covering map between closed surfaces, where $p$ is a positive integer. Given an immersed closed curve $(L,f;F')$, then there is a $p$-fold covering $\tilde L$ of $L$ such that the following commutative diagram holds:
\begin{equation*}
    \xymatrix{
    \tilde L \ar[d]\ar[r]^{\tilde f}& F\ar[d]^g\\
    L \ar[r]^f& F'.
    }
\end{equation*}
The function $\tilde f$ in the diagram is called the {\bf covering lift} of $f$.
\end{lem}

\begin{rmk}
    {The covering lift defined above is the same as the fiber product in category theory.}
\end{rmk}

\begin{proof}
If $L$ is cut open at a point, we get a line segment $[0,1]$. Then $f$ induces a map $f_1$ as in the following diagram:
\begin{equation*}
    \xymatrix{
    L\ar[r]^f&F'\\
    [0,1]\ar[u]\ar[ur]^{f_1}.
    }
\end{equation*}
The map $f_1$ can be lifted into the covering space $F$:
\begin{equation*}
    \xymatrix{
    &F\ar[d]^g\\
    [0,1]\ar[r]^{f_1}\ar[ur]^{\tilde f_1}&F'.
    }
\end{equation*}
Then the deck transformation of the covering map $g$ can be performed to $f_1$ to get $p$ maps
\[
\tilde f_i:[0,1]\rightarrow F,\ \ i=1,\ldots,p.
\]
Finally glue these $p$ maps together according to the deck transformation we have the covering space $\tilde L$ of $L$ and the map $\tilde f$ as in the theorem.
\end{proof}

\begin{lem}\label{lem:idecov}
Let $g:F\rightarrow F'$ be a $p$-fold covering map between closed surfaces, where $p$ is a positive integer. Let $(\tilde L_1,\tilde f_1; F)$ and $(L_2,f_2;F')$ be two immersed curves whose self intersections are transversal. Assume the image of $g\circ\tilde f_1$ misses the self intersections of $L_2$. Denote $\tilde f_2:\tilde L_2\rightarrow F$ as the covering lift of $f_2$. Then the following sets can be identified:
\[
\{(\tilde x_1,x_2)\in\tilde L_1\times L_2|\ g\circ\tilde{f_1}(\tilde x_1)=f_2(x_2)\}
    \cong\{(\tilde x_1,\tilde x_2)\in\tilde L_1\times\tilde L_2|\ \tilde{f_1}(\tilde x_1)=\tilde{f_2}(\tilde x_2)\}.
\]
\end{lem}

\begin{proof}
Let 
\[
\tilde g:\tilde L_2\rightarrow L_2
\]
be the map induced by the covering lift of $\tilde f_2$. This lemma is true if given $x_2\in L_2$, for all $\tilde x_1\in \tilde L_1$ with $g\circ\tilde{f_1}(\tilde x_1)=f_2(x_2)$, then there is a unique $\tilde x_2\in L_2$ such that
\[
\tilde g(\tilde x_2)=x_2\ and\ \tilde{f_1}(\tilde x_1)=\tilde{f_2}(\tilde x_2).
\]
The existence of $\tilde x_2$ is guaranteed by the definition of covering lift. Therefore we only need to show that given $(\tilde x_1,x_2)$ with $g\circ\tilde{f_1}(\tilde x_1)=f_2(x_2)$, if there are $\tilde x_2^1,\tilde x_2^2\in\tilde L_2$ such that 
\[
\tilde g(\tilde x_2^i)=x_2\ and\ \tilde{f_1}(\tilde x_1)=\tilde{f_2}(\tilde x_2^i)\ for\ i=1,2,
\] 
then $\tilde x_2^1=\tilde x_2^2$. Because the image of $g\circ\tilde f_1$ misses the self intersections of $L_2$, the image of $\tilde f_1$ misses the self intersections of $\tilde L_2$. Note that the self intersection of $\tilde L_2$ are transversal since the self intersections of $L_2$ are transversal. Thus if
\[
\tilde g(\tilde x_2^i)=x_2\ and\ \tilde{f_1}(\tilde x_1)=\tilde{f_2}(\tilde x_2^i)\ for\ i=1,2,
\]
then $\tilde x_2^1=\tilde x_2^2$.
\end{proof}

\begin{lem}\label{lem:comps}
    {Let $F$, $F_1$, $F_2$ be three closed surfaces. Denote $\omega_1$, $\omega_2$, $\omega_1\times(-\omega_2)$ as the symplectic forms on $F_1$, $F_2$, $F_1\times F_2$ respectively. Suppose that
\[
g=(g_1,g_2): F\rightarrow F_1\times F_2
\]
is a Lagrangian correspondence from $F_1$ to $F_2$ and both $g_1$ and $g_2$ are covering maps. } Then immersed curves $(L_1.f_1;F_1)$ and $(L_2,f_2;F_2)$ are composable with $g$.
\end{lem}

\begin{proof}
    {Proving that $(L_1,f_1;F_1)$ and $(L_2,f_2;F_2)$ being composable with $g$ is essential the same, so the following only includes the proof of the composability of $(L_1,f_1;F_1)$ with $g$. According to Definition (\ref{def:composable}), we need to show that $f_1\times g$ is transverse to $\Delta_{F_1}\times F_2$, where $\Delta_{F_1}$ is the diagonal of $F_1\times F_1$. This is equivalent to say that
    \[
    \im(df_1)_{x}+\im(dg_1)_y+T_{(f_1(x),g_1(y))}\Delta_{F_1}=T_{(f_1(x),g_1(y))}(F_1\times F_1),
    \]
    for all $x\in L_1$ and $y\in F$ such that $f_1(x)=g_1(y)$. In fact, since the projection of $T_{(f_1(x),g_1(y))}\Delta_{F_1}$ to the tangent space of the first component in $F_1\times F_1$ is surjective and $\im (dg_1)_y$ is the tangent space of the second component in $F_1\times F_1$ (because $g$ is a covering map), therefore the above identity always holds. This finishes the proof of the lemma.}
\end{proof}


\begin{lem}\label{lem:covering}
{Let $F$, $F_1$, $F_2$ be three closed surfaces. Denote $\omega_1$, $\omega_2$, $\omega_1\times(-\omega_2)$ as the symplectic forms on $F_1$, $F_2$, $F_1\times F_2$ respectively. Suppose that
\[
g=(g_1,g_2): F\rightarrow F_1\times F_2
\]
is a Lagrangian correspondence from $F_1$ to $F_2$ and both $g_1$ and $g_2$ are covering maps.} {Given two composable immersed curves $(L_1.f_1;F_1)$ and $(L_2,f_2;F_2)$ $($they are composable with $g$ automaticly$)$}, define 
\[
\tilde f_i:\tilde{L_i}\rightarrow F,i=1,2,
\]
as the covering lift of $f_i$. Then {the bigons of $(L_1\circ F)\times_{F_2}L_2$ and the bigons of $\tilde{L_1}\times_F\tilde{L_2}$} can be identified.
\end{lem}

\begin{proof}
We first identify {the elements of $(L_1\circ F)\times_{F_2}L_2$ and $\tilde{L_1}\times_F\tilde{L_2}$}. From the definition of Lagrangian composition, the Lagrangian immersion $L_1\circ F$ is the composition
\[
g_2\circ\tilde{f_1}:\tilde{L_1}\looparrowright F\stackrel{g_2}\longrightarrow F_2.
\]
To apply the previous lemma, we perturb $L_1$, $L_2$ such that their self intersections are transversal and the image of $g_2\circ\tilde{f_1}$ misses the self intersections of $L_2$. By the definition of fiber products of Lagrangian immersions (\ref{def:lagint}) and Lemma \ref{lem:idecov}, then
\begin{equation}\label{equ:8.1}
    \begin{aligned}
     &(L_1\circ F)\times_{F_2} L_2\\
    =&\{(\tilde x_1,x_2)\in\tilde L_1\times L_2|\ g_2\circ\tilde{f_1}(\tilde x_1)=f_2(x_2)\}\\
    \cong&\{(\tilde x_1,\tilde x_2)\in\tilde L_1\times\tilde L_2|\ \tilde{f_1}(\tilde x_1)=\tilde{f_2}(\tilde x_2)\}\\
    =&\tilde{L_1}\times_F\tilde{L_2}.
    \end{aligned}
\end{equation}
Therefore the above equations give a bijective identification between {$x\in (L_1\circ F)\times_{F_2}L_2$ and $\tilde x\in \tilde{L_1}\times_F\tilde{L_2}$}.\\
Suppose {$D$ is a bigon connecting two points $x$ and $y$ in $(L_1\circ F)\times_{F_2}L_2$.} So the boundary $\partial D$ is a trivial element in the fundamental group of $F_2$, then the lifting property of covering spaces induces a lifted immersed circle in $F$. Moreover this immersed circle bounds a {bigon} $\tilde D$ lifted from $D$ in $F$. Then $\tilde D$ connects {$\tilde x, \tilde y\in\tilde{L_1}\times_F\tilde{L_2}$}, where $\tilde x,\tilde y$ correspond to $x,y$ respectively under the identification of Equation (\ref{equ:8.1}). Since $g_2$ is a covering map, $\tilde D$ has convex corners. Therefore $\tilde D$ is a {bigon of $\tilde{L_1}\times_F\tilde{L_2}$}. This gives a map from the {bigons of $(L_1\circ F)\times_{F_2} L_2$ to the bigons of $\tilde{L_1}\times_F\tilde{L_2}$}. It is clear from the construction that this map is injective. This map is also surjective because one can use $g_2$ to project the {bigons of $\tilde{L_1}\times_F\tilde{L_2}$ to the bigons of $(L_1\circ F)\times_{F_2} L_2$. Thus the bigons of $(L_1\circ F)\times_{F_2} L_2$ and the bigons of $\tilde{L_1}\times_F\tilde{L_2}$ are identified.} The lemma is thus proved.
\end{proof}

\begin{proof}[Proof of Theorem \ref{1}]
{By Lemma \ref{lem:comps}, immersed curves $L_1\looparrowright F_1$ and $L_2\looparrowright F_2$ are composable with the Lagrangian correspondence $F\looparrowright F_1\times F_2$.} According to the previous lemma, the {elements and the bigons of $(L_1\circ F)\times_{F_2} L_2$ and the elements and the bigons of $\tilde{L_1}\times_F\tilde{L_2}$ can be identified respectively}. The same argument shows that the {elements and bigons of $L_1\times_{F_1}(F \circ L_2)$ and the elements and bigons of $\tilde{L_1}\times_F\tilde{L_2}$ are identified respectively. According to the construction of the identification of these bigons, the correspondence from the bigons of $(L_1\circ F)\times_{F_2} L_2$ to the bigons of $L_1\times (F\circ L_2)$ extends the correspondence from $(L_1\circ F)\times_{F_2} L_2$ to $L_1\times (F\circ L_2)$. As a result, the theorem is true.}
\end{proof}

Then consider the case where the Lagrangian immersion $(F,g;F_1\times F_2)$
has bifold points only. According to Theorem \ref{thm:main}, the bisingular set $S$ can be assumed to be a smooth submanifold in $F$. Therefore $S$ is a disjoint union of circles.

\subsection{Lagrangian immersions with bisingularities}\label{subsec:8.2}
Let $g_1: F\rightarrow F_1$ and $g_2:F\rightarrow F_2$ be two covering maps between closed surfaces. Denote $\omega_1$, $\omega_2$, $\omega_1\times(-\omega_2)$ as the symplectic forms on $F_1$, $F_2$, $F_1\times F_2$ respectively. Assume that
\[
(g_1,g_2):F\rightarrow F_1\times F_2
\] 
is a Lagrangian correspondence from $F_1$ to $F_2$. Suppose that $L_1\looparrowright F_1$ and $L_2\looparrowright F_2$ are two immersed curves, then according to Example \ref{exp:2.1}, $L_1$ and $L_2$ are Lagrangian immersions into $F_1$ and $F_2$, resp. This subsection focuses on the case where the Lagrangian immersion $F\looparrowright F_1\times F_2$ only has embedded bifold singularities under the maps $g_1$ and $g_2$.

Recall that in the last subsection, if $g_1$ and $g_2$ have no singularities, {the bigons of $(L_1\circ F)\times_{F_2}L_2$ and the bigon of $L_1\times_{F_1}(F\circ L_2)$} behave well under the Lagrangian composition induced by $F\looparrowright F_1\times F_2$. The examples in this subsection show that if the following map (Definition \ref{def:8.4}) is performed near the bisingularities, then {the bigons of $(L_1\circ F)\times_{F_2}L_2$ and the bigons of $L_1\times_{F_1}(F\circ L_2)$} away from the bisingularities can be identified as in the last subsection. But the Lagrangian composition induced by $F\looparrowright F_1\times F_2$ may provide new {bigons} near the bisingularities.

This subsection starts with comparing {the bigons of $(L_1\circ F)\times_{F_2}L_2$ and the bigons of $L_1\times_{F_1}(F\circ L_2)$}. {Then an example of constructing the bigons of $(L_1\times L_2)\times_{F_1\times F_2}F$ out of $(L_1\circ F)\times_{F_2}L_2$ and {$L_1\times_{F_1}(F\circ L_2)$} is presented.} Finally we {propose a conjecture about} constructing the boundary map of $\CF^Q(L_1,F,L_2;F_1,F_2)$ from {$\CF(L_1\circ F,L_2)$ and $\CF(L_1,F\circ L_2)$}.

\begin{rmk}
    {In \cite{wehrheim2010quilted}, Wehrheim and Woodward proved that under certain monotonicity and admissiblity conditions, the quilted Lagrangian Floer chain group $\CF^Q(L_1,F,L_2;F_1,F_2)$ is a chain complex and its boundary map can be calculated from counting pseduo-holomorphic bigons connecting points in $(L_1\times L_2)\times_{F_1\times F_2}F$.}
\end{rmk}

\begin{defn}\label{def:8.4}
Suppose 
\[
T:S^1\times [-1,1]\rightarrow S^1\times [-1,1]
\]
is a smooth map, where $S^1$ is the unit circle in $\R^2$.
\begin{itemize}
    \item Let $m:[-1,1]\rightarrow[0,1]$ be a smooth map such that $m$ increases strictly in $[-\frac25,\frac25]$, $m(0)=0,$ {$m'(0)\ne0$}, and $m(t)=0$ for $t\in[-1,-\frac35]\cup[\frac35,1]$. A function $T$ is called {\bf a good map} if $T(\exp^{i\theta},t)=(\exp^{i\theta+i\pi m(t)}, t)$, where $\theta\in[0,2\pi)$.
    \item  Let $m:[-1,1]\rightarrow[0,2n]$ be a smooth map such that $m$ increases strictly in $[-\frac14,\frac14]$, $m(0)=0,$ {$m'(0)\ne0$}, $m(t)=0$ for $t\in[-1,-\frac34]$, and $m(t)=2n$ for $t\in[\frac34,1]$. A function $T$ is called an {\bf $n$-Dehn twist} if $T(\exp^{i\theta},t)=(\exp^{i\theta+i\pi m(t)},t)$. Here $n$ is a positive integer.
\end{itemize}
\end{defn}

\begin{rmk}
The good map defined above represents a zero-Dehn twist.
\end{rmk}

\begin{lem}
Let $T$ be a good map or an $n$-Dehn twist. Let $f$ be a fold map defined as below
\begin{align*}
    f: S^1\times[0,1]&\rightarrow S^1\times[0,1]\\
    (\exp^{i\theta},t)&\mapsto(\exp^{i\theta},t^2),
\end{align*}
Note that there is a covering map from $\R\times[-1,1]$ to $S^1\times[-1,1]$. The restriction of the standard symplectic form $\omega_{std}$ of $\R^2$ on $\R\times[-1,1]$ induces a symplectic form $\omega$ on $S^1\times[-1,1]$. Then the map
\[
(f,f\circ T):S^1\times[-1,1]\rightarrow(S^1\times[-1,1])\times (S^1\times[-1,1])
\]
is a Lagrangian immersion into $((S^1\times[-1,1])\times (S^1\times[-1,1]),\omega\times(-\omega))$.
\end{lem}

\begin{proof}
Consider the lift of $(f,f\circ T)$ to the map 
\begin{align*}
  g:\R\times[-1,1]&\rightarrow(\R\times[-1,1])\times(\R\times[-1,1])\\
  (x,t)&\mapsto(x,t^2,x+\pi m(t),t^2).
\end{align*}
The Jacobian of $g$ is
\begin{equation*}
    (dg)_{(x,t)}=
    \begin{bmatrix}
    1&0\\
    0&2t\\
    1&\pi m'(t)\\
    0&2t
    \end{bmatrix}.
\end{equation*}
First we show that $g$ is an immersion. Only need to prove that 
\[
\rank(dg)_{(x,t)}=2\ \text{for}\ \text{all}\ (x,t).
\]
If $t\ne0$, then $\rank(dg)_{(x,t)}=2$. If $t=0$, then 
\begin{equation*}
    (dg)_{(x,0)}=
    \begin{bmatrix}
    1&0\\
    0&0\\
    1&\pi m'(0)\\
    0&0
    \end{bmatrix}.
\end{equation*}
This is a rank two matrix since $m'(0)\ne 0$.\\
It remains to show that $g$ is a Lagrangian. It can be concluded from the following calculation:
\begin{align*}
    g^*(\omega\times(-\omega))&=2tdx\wedge dt-2td(x+\pi m(t))\wedge dt\\
    &=2tdx\wedge dt-2tdx\wedge dt\\
    &=0.
\end{align*}
\end{proof}

Theorems \ref{thm:main} provides an idea of constructing examples of Lagrangian correspondences from one surface to another surface. This helps us to study how Lagrangian immersions change after performing Lagrangian compositions.

Let $F$, $F_1$, $F_2$ be three closed surfaces and $(F,g=(g_1,g_2);F_1\times F_2)$ be a Lagrangian correspondence from $F_1$ to $F_2$ with bisingular points. {Suppose $F_1$, $F_2$ are closed surfaces with the same genus and the map $g_2$ is given by composing $g_1$ with a self-diffeomorphism of $F$. The following is the special case where this self-diffeomorphism uses the maps in Definition (\ref{def:8.4})}\\
Suppose $S\subset F$ is a smooth one dimensional submanifold. Denote $U(S)\subset F$ as a neighbourhood of $S$ {such that} there is an identification
\[
U(S)= S\times [0,1].
\]
Define a smooth map $T:F\rightarrow F$:
\begin{equation*}
    T=\left\{\begin{aligned}
    &{\rm id}\quad\ \ \ \ \ \ \ \ \ \ \ \ \ \ \ \ \ \ \ \ \ \ \ \ \ \ \ \ \ \  \ \ \ \ \ F-U(S)\\
    &{\rm good}\ {\rm map}\ {\rm or}\ n-{\rm Dehn}\ {\rm twist}\quad U(S)
    \end{aligned}
    \right..
\end{equation*}
Therefore $T$ is a self-diffeomorphism of $F$. The Lagrangian immersion $g$ is defined as
\[
g=(g_1,g_1\circ T).
\]
The followings are examples constructed using this idea.

\begin{figure}[H]
\centering 
\includegraphics[width=0.8\textwidth]{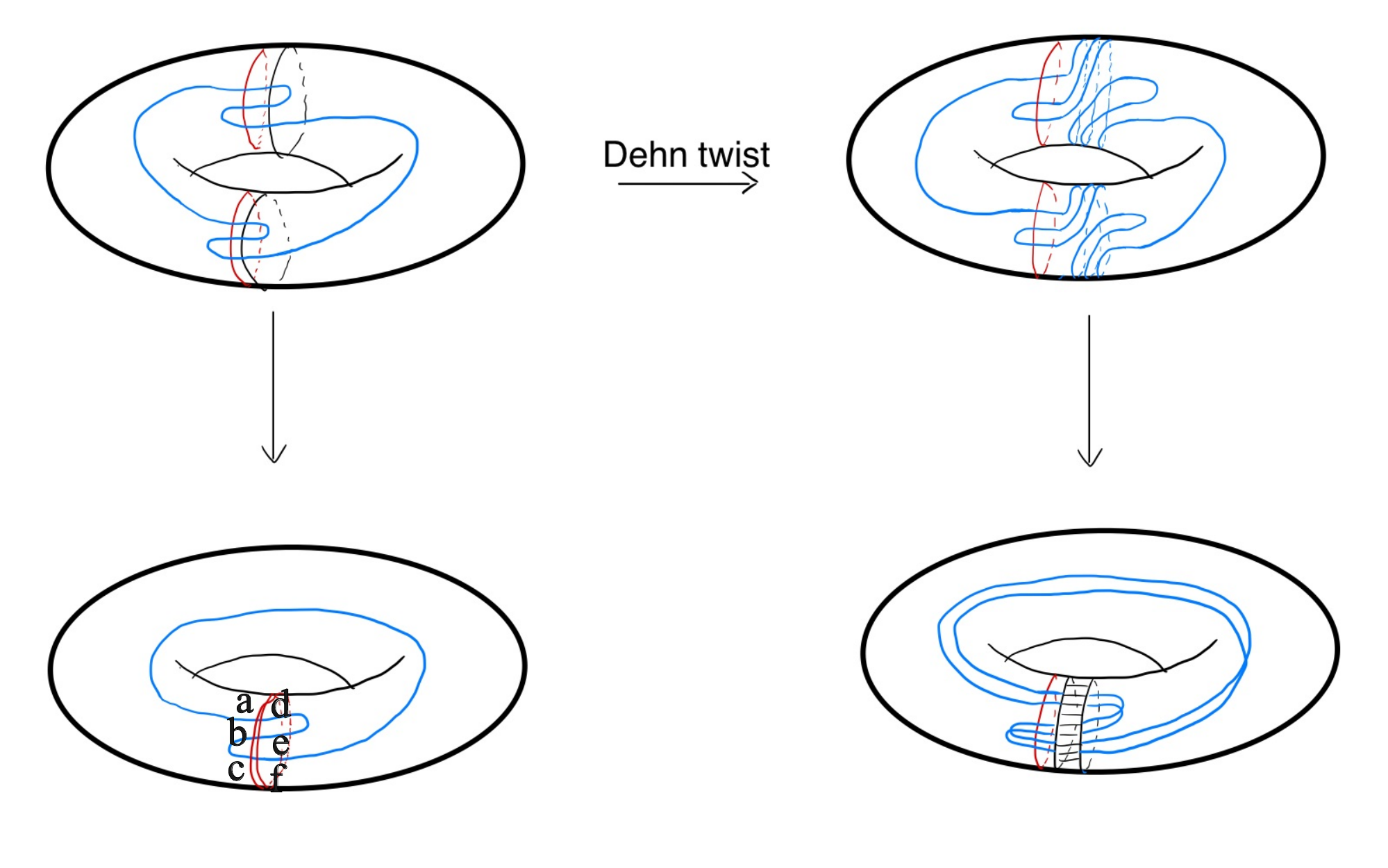}
\caption{The figure for Example \ref{exp:8.10}. The intersections of the blue with the red from the top to the bottom in the lower right picture are $a,d,b,e,c,f$.}
\end{figure}

\begin{ex}\label{exp:8.10}
{This is an example illustrating Theorem \ref{1}. The Lagrangian immersion $g=(g_1,g_2):F\rightarrow F_1\times F_2$ is indicated in Figure 2, where $g_1$ is the left vertical map, and $g_2$ is given by composing $g_1$ with two {1-Dehn} twists as in the top horizontal map. If $g_1$ is a $2:1$ covering map along the longitude, then $g_2$ is also a $2:1$ covering map. The shadow in the lower right picture is the twists induced by the top right picture. The {intersections of the two Lagrangian immersions} on the bottom left and bottom right can be identified. We mark them both as $a,b,c,d,e,f$, as is shown in the picture. Then the {bigon}s connecting these {intersections} are $a$ to $b$, $c$ to $b$, $d$ to $e$, $f$ to $e$. So {the intersections and the bigons} on the bottom left and right are identified.}
\end{ex}

\begin{ex}
The following is an example with bifold singularities. The Lagrangian immersion $g=(g_1,g_2):F\rightarrow F_1\times F_2$ is indicated in Figure 3 below, where $g_1$ is the left vertical map, and $g_2$ is given by composing $g_1$ with two good maps as in the top horizontal map near the bifold singular circles. Here the map $g_1:F\rightarrow F_1$ is given by a projection and the horizontal map is a good map (Definition \ref{def:8.4}) on the bisingular set, then $g_2$ has the same singularity as $g_1$. From the figure, there are exactly two {intersections of the two Lagrangian immersions} in the bottom left picture and a {bigon} connecting them. The same thing is true for the {intersections of the two Lagrangian immersions and the bigon connecting them| corresponding to the bottom right picture.}\\

\begin{figure}[H]\label{fig:F7}
\centering 
\includegraphics[width=0.8\textwidth]{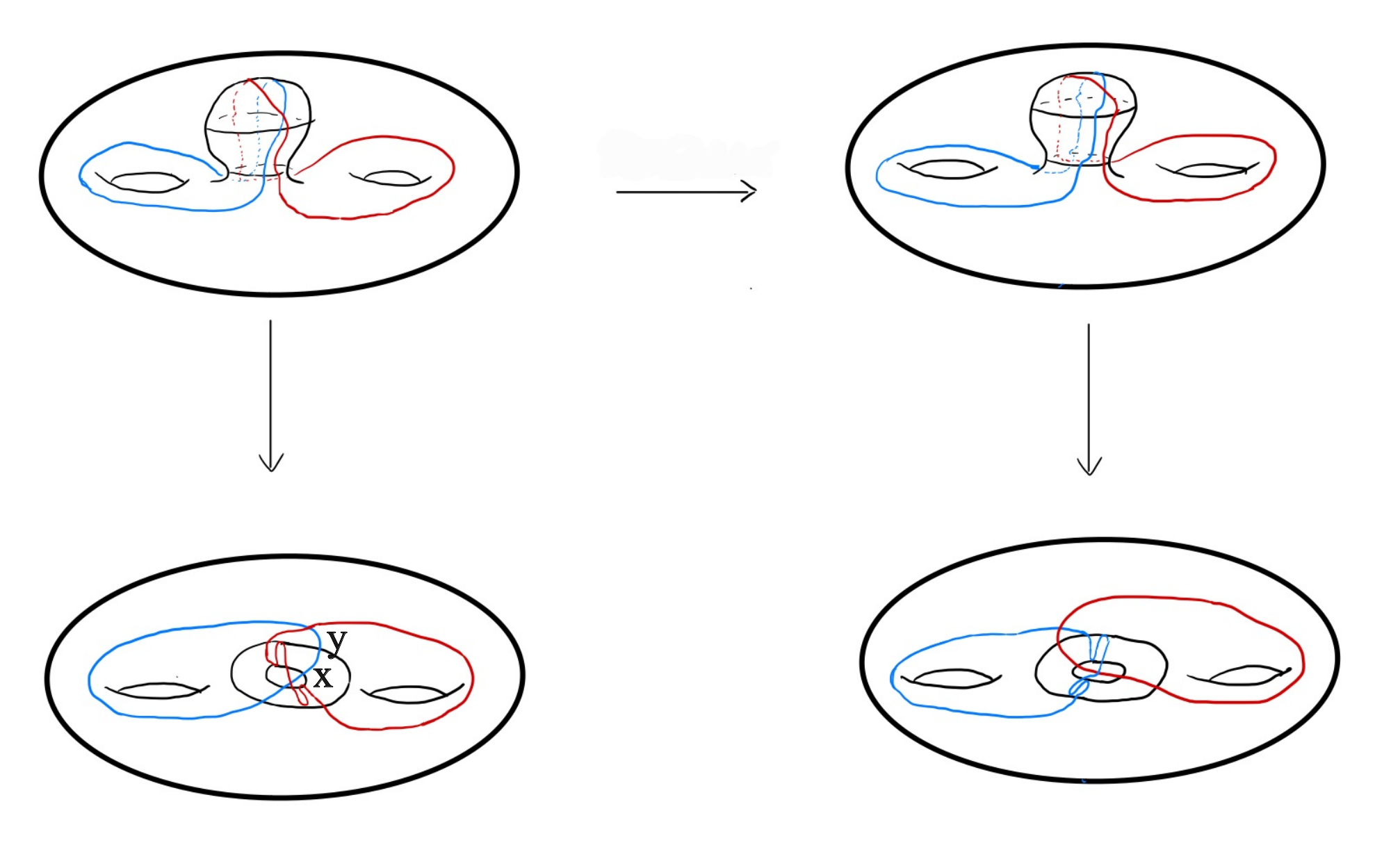}
\caption{The trivial Lagrangian composition in the backside is not drawn} 
\end{figure}

The reason why the conclusion of Theorem \ref{1} still hold is explained as follows. The blue circle in the bottom left picture and the corresponding blue circle on the top left picture can be oriented clockwise. If a point travels along the top left blue circle from the left most point, then this point first hits the red curve, then the smaller fold circle, and finally the large fold circle. Therefore, when performing the Lagrangian composition to the bottom left blue, one needs to consider what happens near $x$ first and then $y$. What the Lagrangian composition looks like around $x$ and $y$ is essentially depicted in Figure 6. These kinds of behaviour of the Lagrangian compositions preserve the bigons.
\end{ex}

\begin{figure}[H]\label{fig:F11}
\centering 
\includegraphics[width=0.5\textwidth]{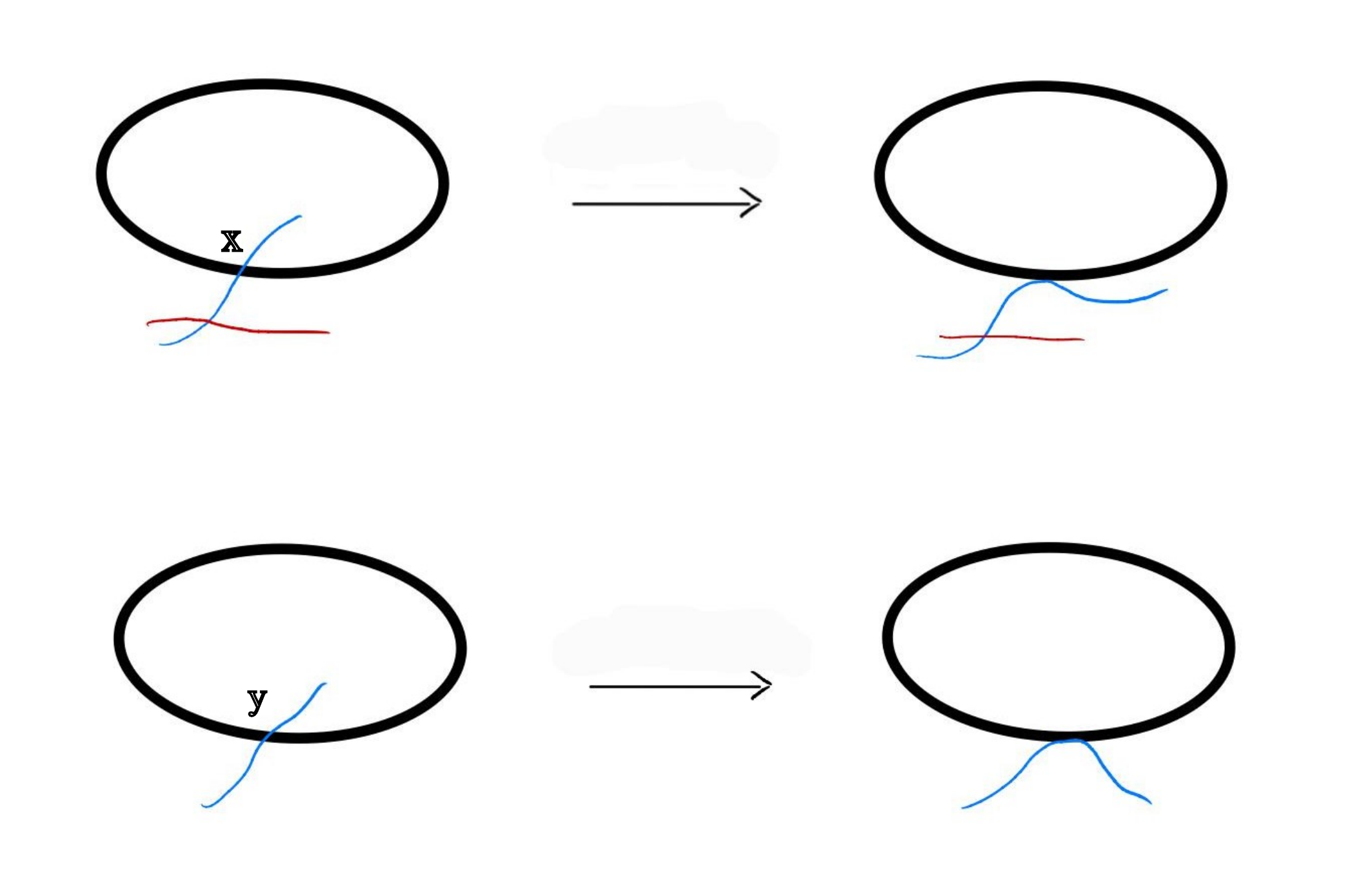}
\caption{{The top line is the Lagrangian composition near $x$. Notice that the blue curve in the top left of Figure 3 first meets the red curve, then meets the small fold circle, finally meets the big fold circle. So after performing the Lagrangian composition, the blue and red curves intersect only once around $x$. The bottom line is the Lagrangian composition near $y$. Notice that the blue curve in the top left of Figure 3 first meets the small fold circle, then meets the big fold circle, finally meets the big fold circle again. So after performing the Lagrangian composition, the blue and red curves do not intersect around $y$.}} 
\end{figure}

\begin{ex}\label{exp:8.12}
{The following is an example with bifold singularities and the Lagrangian Floer chain group vary under a Lagrangian correspondence. The Lagrangian immersion $g=(g_1,g_2):F\rightarrow F_1\times F_2$ is indicated in the figure below, where $g_1$ is the left vertical map, and $g_2$ is given by composing $g_1$ with two 1-Dehn twist as in the top horizontal map near the bifold singular circles. The map $g_1:F\rightarrow F_1$ is given by a fold from top to the bottom, so $g_2$ is has the same fold singularity as $g_1$. From the figure, there are exactly two {intersections $a,b$ of the two Lagrangian immersions} in the bottom left picture without a {bigon} connecting them. But for the {intersections $a,b$ of the two Lagrangian immersions} on bottom right, there is a {bigon} covering the top of the torus. The {non-bijection of these bigons} comes from the Lagrangian composition near the intersection of the red curve with the fold singularities (see middle line of the figure in the next example and compare it with the picture above). }
\begin{figure}[H] 
\centering 
\includegraphics[width=0.5\textwidth]{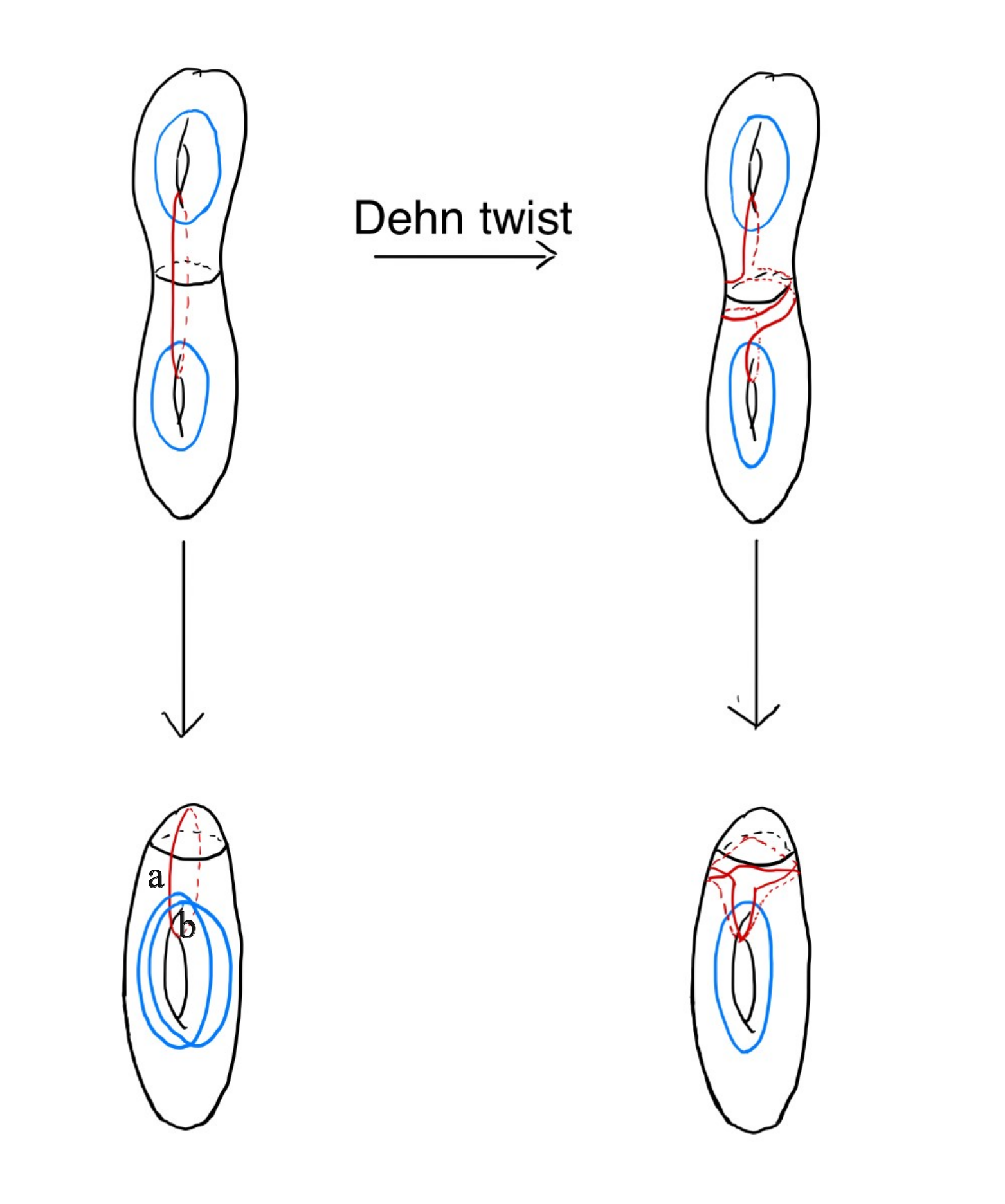}
\end{figure}
\end{ex}

\begin{ex}
The local picture for Lagrangian composition around bisingular sets.
\begin{figure}[H] 
\centering 
\includegraphics[width=0.5\textwidth]{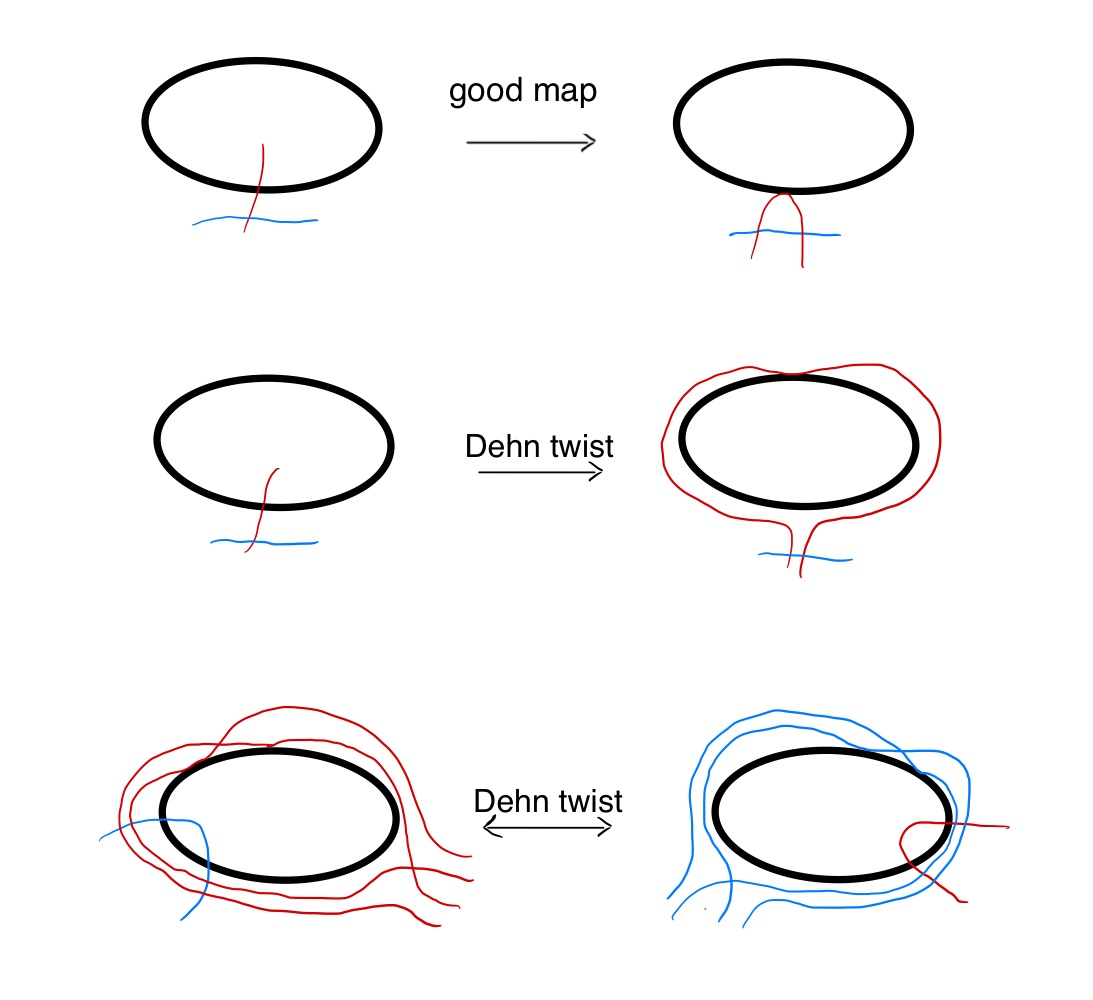}
\end{figure}
\end{ex}

The above examples show that Lagrangian compositions may provide new bigons if there are {bigon}s intersecting the bisingularities.\\

Suppose that all {bigons of {$(L_1\circ F)\times_{F_2}L_2$ and $L_1\times_{F_1}(F\circ L_2)$}} are away from the bisingularities. The following example indicates the possibility of breaking {the bigons of $(L_1\circ F)\times_{F_2}L_2$ and the bigons of $L_1\times_{F_1}(F\circ L_2)$ into bigons of $(L_1\times L_2)\times_{F_1\times F_2}F$}.

\begin{ex}
{The Lagrangian immersion $g=(g_1,g_2):F\rightarrow F_1\times F_2$ is indicated in the figure below, where $g_1$ is the left vertical map, and $g_2$ is given by composing $g_1$ with a 1-Dehn twist as in the top horizontal map near the bifold singular circles. The map $g_1:F\rightarrow F_1$ is given by a fold from top to the bottom, so $g_2$ is has the same fold singularity as $g_1$. The green curve as a map
\[
\R\rightarrow F_1\times F_2
\]
in the following picture can be lifted to $F$ so it breaks the {bigon}s with blue and red boundary in the bottom level into {bigons of $(L_1\times L_2)\times_{F_1\times F_2}F$}, as in the figure below.}
\begin{figure}[H] 
\centering 
\includegraphics[width=0.8\textwidth]{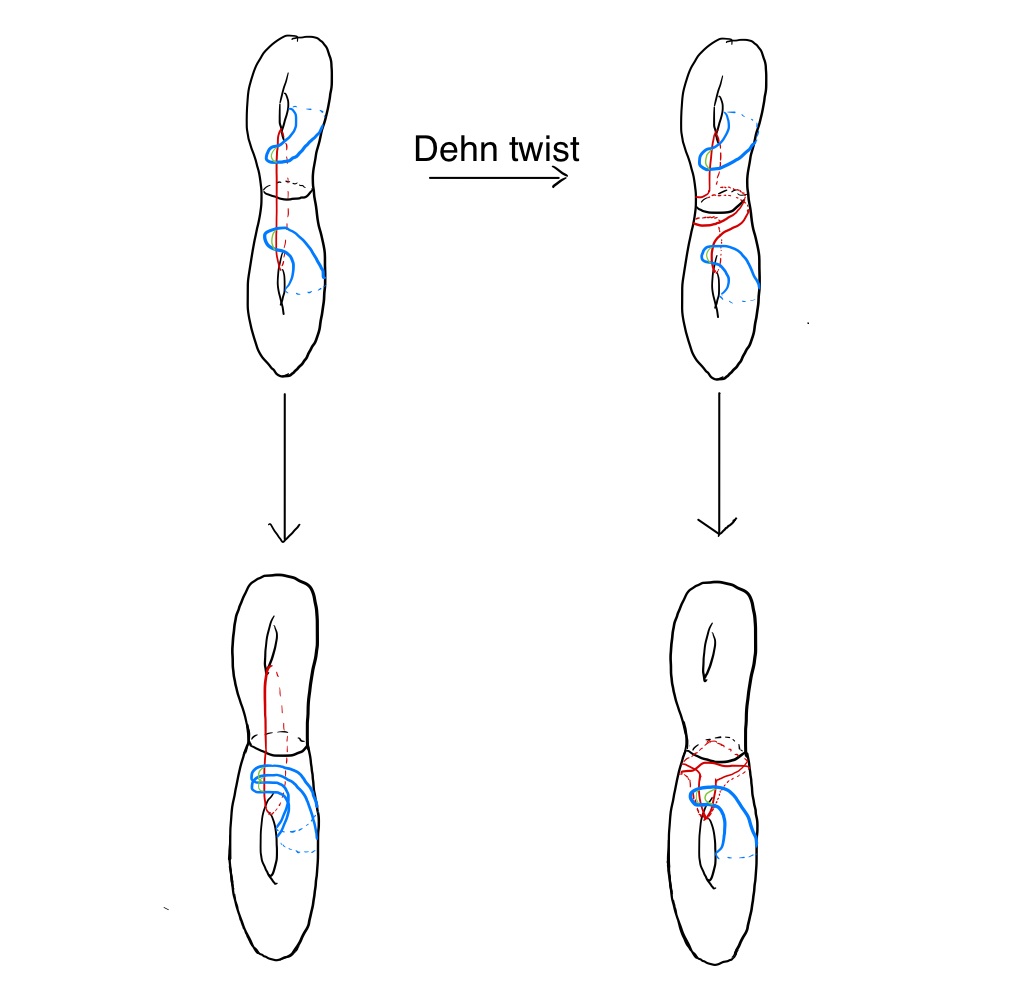}
\end{figure}
\end{ex}
{In the upcoming paper, the author will show that the green curve can be chosen such that one can find a holomorphic {bigon of $(L_1\times L_2)\times_{F_1\times F_2}F$} in this case.}

\begin{conj}
Let $g_1: F\rightarrow F_1$ and $g_2:F\rightarrow F_2$ be two maps with only bifold singularities between closed surfaces. Denote $\omega_1$, $\omega_2$, $\omega_1\times(-\omega_2)$ as the symplectic forms on $F_1$, $F_2$, $F_1\times F_2$ respectively. Assume that
\[
(g_1,g_2):F\rightarrow F_1\times F_2
\] 
is a Lagrangian correspondence from $F_1$ to $F_2$. Suppose that $L_1\looparrowright F_1$ and $L_2\looparrowright F_2$ are two immersed curves, composable with $F$. If the {bigon}s corresponding to the boundary maps in $\CF(L_1\circ F,L_2;F_2)$ and $\CF(L_1,F\circ L_2;F_1)$ are {contained in $\im(g_1)$ and $\im(g_2)$ respectively, then the boundary maps of $\CF^Q(L_1,F,L_2;F_1,F_2)$ can be constructed from $\CF(L_1\circ F,L_2;F_2)$, $\CF(L_1,F\circ L_2;F_1)$.}
\end{conj}

\begin{rmk}
    {The method of Lemma \ref{lem:covering} can not be applied here. The difficulty comes from the Lagrangian composition near the bifold singularities (Example \ref{exp:8.12}), where the covering space argument does not work. Instead, one needs to solve a {Cauchy-Riemann} equation with boundary conditions. In the cases where the Lagrangian Floer chain groups have different boundary maps, e.g. Example \ref{exp:8.12}, there should be a figure-eight bubbling as defined in \cite{bottman2018gromov}.}
\end{rmk}


\nocite{*}
\bibliographystyle{plain}
\bibliography{references}

\begin{thebibliography}{10}

\bibitem{abouzaid2008fukaya}
Mohammed Abouzaid.
\newblock On the {F}ukaya categories of higher genus surfaces.
\newblock {\em Advances in Mathematics}, 217(3):1192--1235, 2008.

\bibitem{akaho2010immersed}
Manabu Akaho and Dominic Joyce.
\newblock Immersed {L}agrangian {F}loer theory.
\newblock {\em Journal of differential geometry}, 86(3):381--500, 2010.

\bibitem{ando1982elimination}
Yoshifumi Ando.
\newblock Elimination of certain {T}hom-{B}oardman singularities of order two.
\newblock {\em Journal of the Mathematical Society of Japan}, 34(2):241--267,
  1982.

\bibitem{arnold1992ordinary}
Vladimir~I Arnold.
\newblock {\em Ordinary differential equations}.
\newblock Springer Science \& Business Media, 1992.

\bibitem{atiyah1988new}
Michael Atiyah.
\newblock New invariants of 3-and 4-dimensional manifolds.
\newblock {\em The mathematical heritage of Hermann Weyl (Durham, NC, 1987)},
  48:285--299, 1988.

\bibitem{bottman2018gromov}
Nathaniel Bottman and Katrin Wehrheim.
\newblock Gromov compactness for squiggly strip shrinking in pseudoholomorphic
  quilts.
\newblock {\em Selecta Mathematica}, 24(4):3381--3443, 2018.

\bibitem{cazassus2020correspondence}
Guillem Cazassus, Christopher~M Herald, Paul Kirk, and Artem Kotelskiy.
\newblock The correspondence induced on the pillowcase by the earring tangle.
\newblock {\em arXiv preprint arXiv:2010.04320}, 2020.

\bibitem{donaldson1996symplectic}
Simon~Kirwan Donaldson.
\newblock Symplectic submanifolds and almost-complex geometry.
\newblock {\em Journal of Differential Geometry}, 44(4):666--705, 1996.

\bibitem{donaldson1983self}
SK~Donaldson.
\newblock Self-dual connections and the topology of smooth 4-manifolds.
\newblock {\em Bulletin (New Series) of the American Mathematical Society},
  8(1):81--83, 1983.

\bibitem{eliavsberg1972surgery}
Ja~M Elia{\v{s}}berg.
\newblock Surgery of singularities of smooth mappings.
\newblock {\em Mathematics of the USSR-Izvestiya}, 6(6):1302, 1972.

\bibitem{eliashberg2002introduction}
Yakov Eliashberg, Nikolai~M Mishachev, and Susumu Ariki.
\newblock {\em Introduction to the $ h $-Principle}.
\newblock Number~48. American Mathematical Soc., 2002.

\bibitem{entov1999surgery}
Mikhail Entov.
\newblock Surgery on {L}agrangian and {L}egendrian singularities.
\newblock {\em Geometric \& Functional Analysis GAFA}, 9(2):298--352, 1999.

\bibitem{evans1997partial}
Lawrence~C Evans.
\newblock Partial differential equations and {M}onge-{K}antorovich mass
  transfer.
\newblock {\em Current developments in mathematics}, 1997(1):65--126, 1997.

\bibitem{floer1988instanton}
Andreas Floer.
\newblock An instanton-invariant for 3-manifolds.
\newblock {\em Communications in mathematical physics}, 118(2):215--240, 1988.

\bibitem{floer1988morse}
Andreas Floer.
\newblock Morse theory for {L}agrangian intersections.
\newblock {\em Journal of differential geometry}, 28(3):513--547, 1988.

\bibitem{freed2012instantons}
Daniel~S Freed and Karen~K Uhlenbeck.
\newblock {\em Instantons and four-manifolds}, volume~1.
\newblock Springer Science \& Business Media, 2012.

\bibitem{fukaya2017unobstructed}
Kenji Fukaya.
\newblock Unobstructed immersed {L}agrangian correspondence and filtered
  ${A}_\infty$ functor.
\newblock {\em arXiv preprint arXiv:1706.02131}, 2017.

\bibitem{fukaya2009lagrangian}
Kenji Fukaya, Yong-Geun Oh, Hiroshi Ohta, and Kaoru Ono.
\newblock Lagrangian intersection {F}loer theory: anomaly and obstruction.
  {P}art i, volume 46 of {AM}{S}.
\newblock {\em IP Studies in Advanced Mathematics. American Mathematical
  Society, Providence, RI}, 2, 2009.

\bibitem{gamelin2003complex}
Theodore Gamelin.
\newblock {\em Complex analysis}.
\newblock Springer Science \& Business Media, 2003.

\bibitem{goldman1984symplectic}
William~M Goldman.
\newblock The symplectic nature of fundamental groups of surfaces.
\newblock {\em Advances in Mathematics}, 54(2):200--225, 1984.

\bibitem{golubitsky2012stable}
Martin Golubitsky and Victor Guillemin.
\newblock {\em Stable mappings and their singularities}, volume~14.
\newblock Springer Science \& Business Media, 2012.

\bibitem{hedden2014pillowcase}
Matthew Hedden, Christopher~M Herald, and Paul Kirk.
\newblock The pillowcase and traceless representations of knot groups ii: a
  {L}agrangian-{F}loer theory in the pillowcase.
\newblock {\em arXiv preprint arXiv:1501.00028}, 2014.

\bibitem{lipshitz2006cylindrical}
Robert Lipshitz.
\newblock A cylindrical reformulation of {H}eegaard {F}loer homology.
\newblock {\em Geometry \& Topology}, 10(2):955--1096, 2006.

\bibitem{mcduff2012j}
Dusa McDuff and Dietmar Salamon.
\newblock {\em J-holomorphic curves and symplectic topology}, volume~52.
\newblock American Mathematical Soc., 2012.

\bibitem{mcduff2017introduction}
Dusa McDuff and Dietmar Salamon.
\newblock {\em Introduction to symplectic topology}, volume~27.
\newblock Oxford University Press, 2017.

\bibitem{oh2015symplectic}
Yong-Geun Oh.
\newblock {\em Symplectic Topology and Floer Homology: {V}olume 2, {F}loer
  Homology and its Applications}, volume~29.
\newblock Cambridge University Press, 2015.

\bibitem{ronga1971calcul}
Felice Ronga.
\newblock Le calcul de la classe de cohomologie entiere duale a {$\Sigma^k$}.
\newblock In {\em Proceedings of Liverpool Singularities—Symposium I}, pages
  313--315. Springer, 1971.

\bibitem{seidel2008fukaya}
Paul Seidel.
\newblock {\em Fukaya categories and {P}icard-{L}efschetz theory}, volume~10.
\newblock European Mathematical Society, 2008.

\bibitem{wehrheim2010functoriality}
Katrin Wehrheim and Chris~T Woodward.
\newblock Functoriality for {L}agrangian correspondences in {F}loer theory.
\newblock {\em Quantum topology}, 1(2):129--170, 2010.

\bibitem{wehrheim2010quilted}
Katrin Wehrheim and Chris~T Woodward.
\newblock Quilted {F}loer cohomology.
\newblock {\em Geometry \& Topology}, 14(2):833--902, 2010.

\bibitem{weinstein1977lectures}
Alan Weinstein.
\newblock {\em Lectures on symplectic manifolds}.
\newblock Number~29. American Mathematical Soc., 1977.

\bibitem{weinstein2006symplectic}
Alan Weinstein.
\newblock The symplectic “category”.
\newblock In {\em Differential Geometric Methods in Mathematical Physics:
  Clausthal 1980 Proceedings of an International Conference Held at the
  Technical University of Clausthal, FRG, July 23--25, 1980}, pages 45--51.
  Springer, 2006.

\bibitem{wells1980differential}
Raymond~O'Neil Wells and Oscar Garc{\'\i}a-Prada.
\newblock {\em Differential analysis on complex manifolds}, volume 21980.
\newblock Springer New York, 1980.

\bibitem{whitney1955singularities}
Hassler Whitney.
\newblock On singularities of mappings of {E}uclidean spaces. i. {M}appings of
  the plane into the plane.
\newblock {\em Annals of Mathematics}, pages 374--410, 1955.

\end{thebibliography}

\end{document}